\documentclass{scrartcl}
\usepackage{amsmath}
\usepackage{amsthm}
\usepackage{amssymb}
\usepackage{dsfont} 
\usepackage[colorlinks]{hyperref}
\usepackage{mdwlist} 
\usepackage{bm} 
\usepackage{tikz}
\usepackage{caption}
\usepackage[round,numbers]{natbib} 
\usepackage{scrlayer-scrpage} 
\usepackage{abstract}
\usepackage{array}
\usepackage{ragged2e} 
\usepackage{booktabs}
\usepackage[all]{xy}
\usepackage{color}

\definecolor{all-matroids}{HTML}{000000}
\definecolor{cofinitary}{HTML}{FF9955}
\definecolor{finitary}{HTML}{808000}
\definecolor{finite}{HTML}{ff0000}
\definecolor{finite-oriented}{HTML}{d35f5f}
\definecolor{fin-cofin-regular}{HTML}{cccccc}
\definecolor{reg-fp-ori}{HTML}{0000ff}
\definecolor{conj-fp-ori}{HTML}{0000ff}
\definecolor{orth-ori}{HTML}{008000}
\definecolor{regular}{HTML}{ff00ff}

\usetikzlibrary{arrows,decorations.markings,positioning}
\tikzset{>=latex,every edge/.append style={thick}}

\theoremstyle{plain}
\newtheorem{thm}{Theorem}[section]
\newtheorem{lem}[thm]{Lemma}
\newtheorem{prop}[thm]{Proposition}
\newtheorem{cor}[thm]{Corollary}
\newtheorem{qstn}[thm]{Open questions}

\theoremstyle{definition}
\newtheorem{defn}[thm]{Definition}
\newtheorem{exmp}[thm]{Example}

\theoremstyle{remark}
\newtheorem{rem}[thm]{Remark}

\newcommand{\NN}{\mathbb{N}}
\newcommand{\ZZ}{\mathbb{Z}}
\newcommand{\RR}{\mathbb{R}}
\newcommand{\BC}{\mathcal{B}}
\newcommand{\CC}{\mathcal{C}}
\newcommand{\IC}{\mathcal{I}}
\newcommand{\MC}{\mathcal{M}}

\newcommand{\SC}{\mathcal{S}}
\newcommand{\TC}{\mathcal{T}}
\newcommand{\UC}{\mathcal{U}}
\newcommand{\VC}{\mathcal{V}}
\newcommand{\WC}{\mathcal{W}}
\newcommand{\TE}{2^E}
\newcommand{\TPME}{2^{\pm E}}

\newcommand{\supp}{\underline}
\DeclareMathOperator{\sep}{sep}

\makeatletter
\def\unbfootnote{\gdef\@thefnmark{}\@footnotetext}
\makeatother

\AtEndDocument{%
  \par
  \vspace{2\bigskipamount}
  \begin{tabular}{p{0.95\textwidth}}%
    \textsc{Nathan Bowler: Universit{\"a}t Hamburg, Department of Mathematics, Bundesstra{\ss}e~55 (Geomatikum), 20146~Hamburg, Germany}\\
    \textit{Email address}: \texttt{nathan.bowler@uni-hamburg.de}
  \end{tabular}
  \par
  \bigskip
  \begin{tabular}{p{0.95\textwidth}}%
    \textsc{Winfried Hochst{\"a}ttler: FernUniversit{\"a}t in Hagen, Fakult{\"a}t f{\"u}r Mathematik und Informatik, 58084~Hagen, Germany}\\
    \textit{Email address}: \texttt{winfried.hochstaettler@fernuni-hagen.de}
  \end{tabular}
  \par
  \bigskip
  \begin{tabular}{p{0.95\textwidth}}%
    \textsc{Stefan Kaspar}\\
    \textit{Email address}: \texttt{stefan.kaspar@tu-dortmund.de}
  \end{tabular}}

\begin{document}

\author{Nathan Bowler, Winfried Hochst{\"a}ttler, Stefan Kaspar}
\title{On the Possibilities of Defining Infinite Oriented Matroids}
\date{}
\maketitle
\unbfootnote{2020 \textit{Mathematics Subject Classification}. Primary 05B35; Secondary 05C63}
\unbfootnote{\textit{Keywords}. Cryptomorphic axiom systems, infinite matroids, infinite oriented matroids}

\begin{abstract}
  Is it possible to define cryptomorphic axiom systems for infinite oriented matroids by lifting
  some of the axiom systems for finite oriented matroids to the infinite setting while not losing
  duality in the process? We show that the answer to this question is a twofold "no". First, lifting
  the circuit axioms neither preserves duality nor inheritance of strong circuit elimination in
  minors. Second, although duality is kept intact by translating the orthogonality axioms and an
  axiom system based on the Farkas Lemma, the classes of infinite oriented matroids obtained in
  this way have the property that one is a proper subclass of the other.
\end{abstract}

\section{Introduction}
The recent rediscovery of a more general concept of infinite matroids by \citet{infmatroids},
first investigated by \citet{matroidsandduality}, has led to several cryptomorphic axiom systems
for matroids with the following satisfactory properties: they not only extend the theory of finite
matroids in a natural way to the infinite setting but also allow for duality, one of the key concepts
of finite matroids. In particular, this addresses one of the problems that traditional approaches
to infinite matroids based on so called {\em (co)finitary} matroids exhibit: such matroids may only
have infinite cocircuits or infinite circuits, but not both, breaking duality. It is noteworthy that
Dress addressed and solved this particular problem, among other things, in
\citep{dualtheofininfinmatswcoeffs} by a less general approach to infinite matroids named
{\em matroids with coefficients} in between the works of Higgs and Bruhn et al. Both theories are
linked via so called {\em tame matroids}, that is matroids in the sense of Bruhn et al. that have
the property that any circuit-cocircuit intersection is finite: a given matroid in the sense of
Bruhn et al. corresponds to a matroid with coefficients if and only if it is tame (see for instance
\citet{infmatroids} and \citet{strongdualpmatswithcoeff}).

Like matroids, finite oriented matroids are characterized by a broad variety of cryptomorphic
axiom systems (see for instance \citep{orimatsbook}). It is thus natural to ask whether some
or all of these can be lifted to the infinite setting by basing them on this more general
concept of infinite matroids by Higgs and Bruhn et al. This question is especially relevant
since some of the recent efforts to develop a theory of infinite oriented matroids as well as
their applications still lack the presence of duality (see for instance \citet{finaffinorimats}).
Ideally, one would obtain several cryptomorphic axiom systems for infinite oriented matroids
in this process. However, in this paper we show that this approach fails in several ways:
On the one hand, by merely translating an axiom system to the infinite setting, one does not always
obtain a sensible axiom system for infinite oriented matroids. On the other hand, it is possible to
state axiom systems that yield distinct classes of infinite oriented matroids such that one is a
proper subclass of the other.

As for the first claim, we show that translating the so called circuit axioms for oriented
matroids \citep[Definition 3.2.1]{orimatsbook} to the infinite setting does not lead to a sensible
axiom system for infinite oriented matroids, even if the so called axiom of strong signed circuit
elimination \citep[Theorem 3.2.5]{orimatsbook} is incorporated accordingly. This is due to the fact
that unlike in the finite case strong signed circuit elimination fails to be inherited by minors and
also does not carry over to duals. To assert the validity of these claims, we provide an example
of a matroid together with a suitable pair of signatures for its circuits and cocircuits
and show that it exhibits the desired properties (Example~\ref{exmp:OrthOriMAtWithSSCEAndNonSSCECMinor}).
In particular, one does not have to reach far out into the realm of infinite matroids to find
counterexamples like these since the matroid considered is a finitary one.

To address the second claim, we present two distinct axiom systems for infinite oriented matroids
both of which are compatible with duality and yield classes of infinite oriented matroids that are
closed under taking minors. The first of these axiom systems is a straightforward translation of the
so called orthogonality axioms
for finite oriented matroids \citep[Theorem 3.4.3]{orimatsbook} to the infinite setting. Accordingly,
we call the oriented matroids arising from this axiom system {\em orthogonally orientable}. The class
of orthogonally orientable matroids has duality built in; it is also closed under taking minors
(Theorem~\ref{thm:OrthImplOrthMinors}). Modifying
the prototypical forbidden minors characterization of regular tame matroids given by Bowler and Carmesin
in \citep{exclminorinfmat}, we obtain a forbidden minors characterization of tame orthogonally orientable
matroids (Theorem~\ref{thm:OrthCharByMinors}). This result is in part
due to the very general nature of the orthogonality axioms. The property that a pair of circuit and
cocircuit signatures of an orthogonally orientable matroid has to satisfy is a pure abstraction of
vector space orthogonality that does not refer to additional, stronger properties like (strong) signed
circuit elimination. In the context of finite oriented matroids, this fact has no impact on the different
axiom systems being cryptomorphic. However, with respect
to the theory of infinite oriented matroids it does: strong signed circuit elimination still implies
orthogonality (Proposition~\ref{prop:SignedCircElimImplOrthAxioms}) but the inverse is no longer true
(Example~\ref{exmp:OrthOriMAtWithSSCEAndNonSSCECMinor}). This behavior is also observed by Wenzel in
the context of infinite matroids with coefficients in a finitary fuzzy ring
\citep{dpairsmatscoeffinfuzzyroar}. In this paper, continuing the work of \citet{dualtheofininfinmatswcoeffs},
Wenzel extends Dress' approach to several classes of matroids which need not necessarily be tame. In
particular, he presents a definition of infinite oriented matroids with coefficients in a finitary fuzzy
ring which is also a direct translation of the orthogonality axioms to the infinite setting
\citep[Definition 5.7]{dpairsmatscoeffinfuzzyroar}. This alternative approach to infinite oriented
matroids is thus compatible with our definition of orthogonally orientable matroids. Notably, it also
refers only to orthogonality and does not incorporate additional, stronger properties like (strong)
signed circuit elimination.

The second axiom system for infinite oriented matroids we present is based on the so called {\em Farkas Lemma}
for oriented matroids. This lemma states that every element of an oriented matroid is either contained in a
positive circuit or a positive cocircuit, but not both. We call oriented matroids obtained by this axiom system
{\em FP-orientable}. Our definition of FP-orientable matroids is based on the definition of finite oriented
matroids given by Bachem and Kern in \citep{bachemlpd}, especially \citep[Definition 5.8]{bachemlpd}. As in
the case of orthogonally orientable matroids, duality is already a part of the definition. However, to obtain
the result that the class of FP-orientable matroids is closed under taking minors
(Theorem~\ref{thm:MinorsFPAxiomsOMsAreFPAxiomsOMs}), much stronger requirements than in the finite case must
be posed on the pairs of circuit and cocircuit signatures considered. Consequently, a pair of circuit and
cocircuit signatures belonging to an FP-orientable matroid possesses much stronger properties than one
belonging solely to an orthogonally orientable matroid. For instance, strong signed (co)circuit elimination
is possible for the (co)circuits of an FP-orientable matroid (Corollary~\ref{cor:FPOrientationsImplSSCE})
and a rough analog of the so called {\em Composition Theorem} \citep[Theorem 5.36]{bachemlpd} holds with
respect to its (co)vectors (Corollary~\ref{cor:FPOriMatsAllowConfDecomp}). Examples of FP-orientable matroids
include algebraic cycle matroids (Example~\ref{exmp:AlgCycleMatsAreFPOriMats}) and (co)finitary regular matroids
\citep[Section 5]{circpartinfmats}, see also Example~\ref{exmp:CofinRegMatsAreFPOriMats}. Due to the stronger
requirements for FP-orientability, it is to be expected that not all orthogonally orientable matroids are
FP-orientable. Indeed, subsequently we show that the class of FP-orientable matroids is a proper subclass of the
class of orthogonally orientable matroids (Examples~\ref{exmp:OrthOriMAtWithSSCEAndNonSSCECMinorCont} and
\ref{exmp:CMWithoutFarkasCont}), thus establishing the validity of our second claim.

This paper is organized as follows. In Section~\ref{sec:prelim}, we provide an overview of the
notation and basic results with respect to (oriented) matroid theory that we will use. We then
state the axiom system for orthogonally orientable matroids and examine their properties
in Section~\ref{sec:OrthAxioms}. Note that this section is a streamlined version
of \citep[Section 3]{orthaxiomsinforimats} written by the last two authors. In
Section~\ref{sec:NoPlainCircAxioms}, we consider the axiom of strong signed circuit elimination
and investigate why it fails to be inherited by minors and also does not carry over to duals.
Additionally, we look at some of the links between (strong) signed circuit elimination and
orthogonality. In Section~\ref{sec:FPAxsioms}, we introduce the axiom system for FP-orientable
matroids alongside some technical supplements to simplify the proofs of statements about such
matroids. We then examine the properties of FP-orientable matroids as well as examples and
counterexamples. Finally, we conclude this paper with Section~\ref{sec:OverviewClassesOfOriMats}
where we provide a visual overview of the classes of (oriented) matroids that we refer to
throughout the text.

\section{Preliminaries}\label{sec:prelim}
\subsection{Notation and Terminology}\label{ssec:NotationAndTerminology}
Any notation and terminology regarding matroids and oriented matroids not explained below is
taken from \citet{oxley} and \citet{orimatsbook}, respectively.\\

In the following, we always denote a {\em matroid} by $M$, its dual matroid by $M^*$, and
its finite or infinite {\em ground set} by $E$. If we want to emphasize to which matroid $E$
belongs, then we write $E(M)$ instead of just $E$. If $X \subseteq E$, then we denote its
complement $E \setminus X$ by $\overline{X}$. If $X \subseteq E$ and $e \in E$, then we
abbreviate $X \cup \{e\}$ to $X \cup e$ and $X \setminus \{e\}$ to $X \setminus e$. We denote
the {\em power set} of $E$ by $\TE$.\\

A {\em signed subset} $X$ of a set $E$ is a subset $\supp{X}$
of $E$ together with a partition $(X^+, X^-)$ of $\supp{X}$ where $X^+$ contains the so called
{\em positive elements} of $X$ and $X^-$ the so called {\em negative elements} of $X$. The
set $\supp{X}$ is called the {\em support} of $X$. If $A \subseteq E$ and $X$ is a signed subset
of $E$, then the {\em restriction} of $X$ to $A$ is the signed subset $X|_A$ where
$X|_A^+ = X^+ \cap A$ and $X|_A^- = X^- \cap A$. A signed subset $Y$ of $E$ {\em conforms} to a
signed subset $X$ of $E$ if $Y$ is the restriction of $X$ to $\supp{Y}$. The {\em set of all
signed subsets} of $E$ is denoted by $\TPME$. If $X$ is a signed subset of $E$, then we write
$X(e) = 1$ if $e \in X^+$ and $X(e) = -1$ if $e \in X^-$. If $\supp{X} \ne \emptyset$, then we say
that $X$ is {\em positive} or {\em negative} if $\supp{X} = X^+$ or $\supp{X} = X^-$, respectively.
Two signed subsets $X, Y$ of $E$ are said to be {\em orthogonal} if either
$\supp{X} \cap \supp{Y} = \emptyset$, or the restrictions of $X$ and $Y$ to their intersection are
neither equal nor opposite, i.e., there are $e,f \in \supp{X} \cap \supp{Y}$ such that
$X(e)Y(e) = -X(f)Y(f)$. The {\em separator} of two signed subsets $X, Y$ of $E$ is defined as
$\sep(X,Y) = (X^+ \cap Y^-) \cup (X^- \cap Y^+)$. The {\em opposite} of a signed subset $X$ of $E$
is the signed subset $-X$ where $-X^+ = X^-$ and $-X^- = X^+$. A set of signed subsets
$\SC \subseteq \TPME$ is called {\em symmetric} if $\SC = \{ -X\ |\ X \in \SC \}$. If $X$ is a signed
subset of $E$ and $A \subseteq E$, then we say that the signed subset $_{-A}X$ where
$_{-A}X^+ = (X^+ \setminus A) \cup (X^- \cap A)$ and $_{-A}X^- = (X^- \setminus A) \cup (X^+ \cap A)$
is obtained from $X$ by {\em reorientation} on $A$. Finally, if $\SC \subseteq \TPME$, then we let
$_{-A}\SC = \{ _{-A}X\ |\ X \in \SC \}$.\\

A {\em circuit signature} $\CC$ of a matroid $M$ assigns to each circuit $\supp{C}$ of $M$ two
opposite signed subsets $C$ and $-C$ of $E$ supported by $\supp{C}$. A circuit signature $\CC^*$ of
$M^*$ is also called a {\em cocircuit signature} of $M$.\\

A {\em finite} (or {\em ordinary}) {\em oriented matroid} $\MC$ is a triple $(M, \CC, \CC^*)$
that consists of an ordinary matroid $M$ and a pair of circuit and cocircuit signatures $\CC$
and $\CC^*$ of $M$ that has the property that every signed circuit is orthogonal to every signed
cocircuit. The matroid $M$ is called the {\em underlying matroid} of $\MC$.\\

By $\NN$ we denote the set of natural numbers not including 0 and let $\NN_0 = \NN \cup \{ 0 \}$;
the set of integers and the set of real numbers are denoted by $\ZZ$ and $\RR$, respectively.

\subsection{Infinite Matroids}
We briefly recall two of the cryptomorphic definitions of an (infinite) matroid as given by Bruhn et al.
in \citep{infmatroids} as well as some other basic definitions and state some basic results about (infinite)
matroids.

\begin{defn}\label{defn:MatIndAxioms}
  A set $\IC \subseteq \TE$ is the set of independent sets of a matroid $M$ on a set $E$ if and only
  if it satisfies the following \textbf{independence axioms}:
  \begin{enumerate*}
    \item[(I1)] $\emptyset \in \IC$.
    \item[(I2)] $\IC$ is closed under takings subsets.
    \item[(I3)] For all $I \in \IC \setminus \IC^{\max}$ and $I' \in \IC^{\max}$ there is an
    $x \in I' \setminus I$ such that $I + x \in \IC$.
    \item[(IM)] Whenever $I \subseteq X \subseteq E$ and $I \in \IC$, the set
    $\{ I' \in \IC\ |\ I \subseteq I' \subseteq X \}$ has a maximal element.
  \end{enumerate*}
\end{defn}

\begin{defn}\label{defn:MatCircAxioms}
  A set $\CC \subseteq \TE$ is the set of circuits of a matroid $M$ on a set $E$ if and only if it
  satisfies the following \textbf{circuit axioms}:
  \begin{enumerate*}
    \item[(C1)] $\emptyset \notin \CC$.
    \item[(C2)] No element of $\CC$ is a subset of another.
    \item[(C3)] Whenever $X \subseteq C \in \CC$ and $(C_x\ |\ x \in X)$ is a family
    of elements of $\CC$ such that $x \in C_y \Leftrightarrow x = y$ for all $x,y \in X$,
    then for every $f \in C \setminus \left( \bigcup_{x \in X} C_x \right)$ there exists
    an element $D \in \CC$ such that $f \in D \subseteq \left(C
    \cup \bigcup_{x \in X} C_x \right) \setminus X$.
    \item[(CM)] The set $\IC$ of all $\CC$-independent sets satisfies (IM). These are the
    sets $I \subseteq E$ such that $C \nsubseteq I$ for all $C \in \CC$.
  \end{enumerate*}
  If we want to emphasize to which matroid the set of circuits $\CC$ belongs, then we
  write $\CC(M)$ instead of just $\CC$.
\end{defn}

\begin{thm}
  The independence axioms and the circuit axioms are cryptomorphic:
  \begin{enumerate*}
    \item If a set $\IC \subseteq 2^E$ satisfies the independence axioms, then the set $\CC$ of
      circuits satisfies the circuit axioms with $\IC$ as the set of $\CC$-independent sets.
    \item If a set $\CC \subseteq 2^E$ satisfies the circuit axioms, then the set $\IC$ of
      $\CC$-independent sets satisfies the independence axioms with $\CC$ as the set of circuits.
  \end{enumerate*}
\end{thm}
\begin{proof}
  See \citep[Theorem 4.3]{infmatroids}.
\end{proof}

Since circuits and cocircuits infinite matroids can be infinite, it is custom to make the following
distinctions.

\begin{defn}
  A matroid is called \textbf{finitary} if any set whose finite subsets are independent is also
  independent. A \textbf{cofinitary} matroid is the dual of a finitary matroid.
\end{defn}

\begin{prop}
  A matroid is finitary if and only if every circuit is finite.
\end{prop}
\begin{proof}
  See \citep[Corollary 3.9]{infmatroids}.
\end{proof}

\begin{defn}
  A matroid is called \textbf{tame} if any circuit-cocircuit intersection is finite.
  Otherwise it is called \textbf{wild}.
\end{defn}

What wild matroids can look like and how to construct certain types of them is detailed by
Bowler and Carmesin in \citep{matswithinfcircocircinter}, for instance.

\begin{exmp}
One type of tame matroids that we will encounter several times in the course of this work are the
so-called {\em uniform matroids of rank $n$}, where $n \in \NN$. Such matroids are finitary and
their cocircuits are the sets missing exactly $n-1$ points.
\end{exmp}

\begin{exmp}\label{exmp:TameRegMats}
  In \citep{exclminorinfmat}, the notion of finite {\em regular} matroids is extended to tame matroids.
  One of the given characterizations states that a tame matroid $M$ is called regular if and only if it
  is {\em signable}, i.e.\ for each circuit $C$ of $M$ and for each cocircuit $U$ of $M$ there exists
  a choice of functions $f_C \colon C \rightarrow \{-1,1\}$ and $g_U \colon U \rightarrow \{-1,1\}$,
  respectively, such that
  \[
    \sum_{e \in C \cap U} f_C(e) g_U(e) = 0,
  \]
  where the sum is evaluated over $\ZZ$. Tame signable matroids bear roughly the same relation to
  ordinary tame matroids that directed graphs do to graphs. Examples of tame regular/signable matroids
  are the finite cycle matroid, the algebraic cycle matroid, and the topological cycle matroid of a
  given graph (see for instance \citep[Subsection 5.3]{exclminorinfmat}).
\end{exmp}

Like in the finite case, the set of circuits of a contraction of a matroid can be directly
characterized as follows.

\begin{prop}\label{prop:CharNOCMinor}
  Let $M$ be a matroid on a set $E$ and let $X \subseteq E$. Then the set of circuits $\CC(M.X)$ of
  the contraction of $M$ to $X$ is given by
  \[
    \CC(M.X) = \min(\{ C \cap X\ |\ C \in \CC(M), C \cap X \ne \emptyset \}).
  \]
\end{prop}
\begin{proof}
  "$\subseteq$": Let $C' \in \CC(M.X)$. Then by \citep[Lemma 2.3]{exclminorinfmat} there exists
  a $C \in \CC(M)$ such that $C' \subseteq C \subseteq C' \cup \overline{X}$. This implies
  $C \cap X = C'$. Since $C'' \cap X$ is dependent in $M.X$ for any $C'' \in \CC(M)$
  by \citep[Corollary 3.6]{infmatroids} whenever $C'' \cap X \ne \emptyset$, there cannot
  exist a $C'' \in \CC(M)$ such that $\emptyset \ne C'' \cap X \subset C \cap X = C'$.
  Thus $C \cap X$ is minimal.\\
  "$\supseteq$": Let $C \in \CC(M)$ such that $\emptyset \ne C \cap X$ is minimal. Then by
  \citep[Corollary 3.6]{infmatroids} $C \cap X$ is dependent in $M.X$ and there exists a
  $C' \in \CC(M.X)$ such that $C' \subseteq C \cap X$. As shown in "$\subseteq$", there exists a
  $C'' \in \CC(M)$ such that $C' = C'' \cap X \subseteq C \cap X$. Since $C \cap X$ is
  minimal, this implies that $C' = C \cap X$.
\end{proof}

\begin{lem}\label{lem:StructCircOfMinors}
  Let $M$ be a matroid on $E = X \dot{\cup} F \dot{\cup} G$ and let $C'$ be a circuit of
  the minor $M\!/F\!\setminus\!G$. Then there is a circuit $C$ of $M$ such that
  $C' \subseteq C \subseteq C' \cup F$.
\end{lem}
\begin{proof}
  See \citep[Lemma 2.3]{exclminorinfmat}.
\end{proof}

The following well-known property of circuits and cocircuits of finite matroids carries
over to the circuits and cocircuits of infinite matroids.

\begin{lem}\label{lem:2ElemCircCocircInters}
  Let $M$ be a matroid on $E$ and $C \in \CC$. For any two elements $e,f \in C$, there is a
  cocircuit $D$ of $M$ such that $C \cap D = \{e,f\}$.
\end{lem}
\begin{proof}
  See \citep[Lemma 2.2]{exclminorinfmat}.
\end{proof}

Unions of circuits or cocircuits of a matroid are sometimes of special interest and thus are given
a dedicated name to make referring to them more easy.

\begin{defn}
  A \textbf{scrawl} of a matroid $M$ is a union of circuits of $M$. Likewise, a \textbf{coscrawl}
  of a matroid $M$ is a union of cocircuits of $M$. Coscrawls are also often referred to as
  {\em open sets}.
\end{defn}

When working with scrawls, the following results are useful.

\begin{lem}\label{lem:ScrawlChar}
  Let $M$ be a matroid on $E$ and $V \subseteq E$. Then the following are equivalent:
  \begin{enumerate*}
    \item $V$ is a scrawl of $M$.
    \item $V$ never meets a cocircuit of $M$ just once.
    \item $V$ never meets a coscrawl of $M$ just once.
  \end{enumerate*}
\end{lem}
\begin{proof}
  See \citep[Lemma 2.6]{exclminorinfmat}.
\end{proof}

\begin{cor}\label{cor:StructScrawlsOfMinor}
  Let $M$ be a matroid on $E = X \dot{\cup} F \dot{\cup} G$ and let $V' \subseteq X$. Then
  $V'$ is a scrawl of the minor $M\!/F\!\setminus\!G$ if and only if there is a scrawl $V$
  of $M$ with $V' \subseteq V \subseteq V' \cup F$.
\end{cor}
\begin{proof}
  See \citep[Corollary 2.7]{exclminorinfmat}.
\end{proof}

\section{Orthogonality Axioms}\label{sec:OrthAxioms}
Any sensible definition of infinite oriented matroids encloses the following two
theoretical aspects: First, the definition should enable the concept of duality, i.e.\ it
should be possible to define and conclude the existence of the dual of an infinite
oriented matroid. Second, for a class of oriented matroids to be closed under taking
minors, it is essential that the signature of the circuits and cocircuits of an oriented
matroid is passed on to the circuits and cocircuits of its minors in a unique and compatible
way. Keeping these goals in mind, we begin our investigation of possible definitions of infinite
oriented matroids by examining how the so called orthogonality axioms (see for instance
\citet[Section 3.4]{orimatsbook}) behave in the infinite setting. As we shall see, the definitions
involved in those axioms extend to infinite sets and infinite signed subsets without further
modification. Even more importantly, orthogonality allows the definition of a class of infinite
oriented matroids which explicitly supports duality and is closed under taking minors.

\subsection{Orthogonally Orientable Matroids}
Recall from Section~\ref{ssec:NotationAndTerminology} that orthogonality as an abstract property
of signed subsets is defined as follows.

\begin{defn}\label{defn:OrthSS}
  Two signed subsets $X, Y$ are said to be \textbf{orthogonal}, denoted by $X \perp Y$, if either
  $\supp{X} \cap \supp{Y} = \emptyset$, or the restrictions of $X$ and $Y$ to their
  intersection are neither equal nor opposite, i.e., there are $e,f \in \supp{X} \cap
  \supp{Y}$ such that $X(e)Y(e) = -X(f)Y(f)$.
\end{defn}

Like in the finite case, matroids whose circuits and cocircuits can be signed in such a way that
the corresponding circuit and cocircuit signatures satisfy the orthogonality property from
Definition~\ref{defn:OrthSS} are of special interest.

\begin{defn}\label{defn:OrthOrient}
  Let $M$ be a matroid on a set $E$. Let $\CC$ be a circuit signature of $M$ and $\CC^*$ be a cocircuit
  signature of $M$. Then the triple $\MC = (M, \CC, \CC^*)$ is called an
  \textbf{orthogonally oriented matroid (on $\bm{E}$)} if and only if $\CC$ and
  $\CC^*$ satisfy one of the following two equivalent conditions:
  \begin{itemize*}
    \item[(O)] $\forall C \in \CC, U \in \CC^*: C \perp U$,
    \item[(O')] $\forall C \in \CC, U \in \CC^*: \sep(C, U) \ne \emptyset
      \Leftrightarrow \sep(C, -U) \ne \emptyset$.
  \end{itemize*}
  In this case we say that $M$ is \textbf{orthogonally orientable (on $\bm{E}$)} and that the pair
  $\CC, \CC^*$ \textbf{provides an orthogonal orientation of $\bm{M}$ (on $\bm{E}$)}. The matroid
  $M$ is also referred to as the \textbf{underlying matroid of $\bm{\MC}$}.
\end{defn}

\begin{exmp}\label{exmp:FinOriMatsAreOrthOriMats}
  Since Definition~\ref{defn:OrthOrient} corresponds exactly to the definition of finite oriented
  matroids when $E$ is restricted to finite sets, all finite oriented matroids obviously are orthogonally
  oriented matroids (see \citet[Theorem 3.4.3]{orimatsbook})).
\end{exmp}

As one would expect, matroids that are induced by (infinite) directed graphs are orthogonally
orientable matroids, as the following example explains.

\begin{exmp}\label{exmp:TameRegMatsAreOrthOriMats}
  Recall from Example~\ref{exmp:TameRegMats} that every tame regular matroid is signable. The signing
  of such a matroid $M$ induces a circuit signature $\CC$ of $M$ if we also take the opposite of any
  signed circuit into consideration. In the same way, we obtain a cocircuit signature $\CC^*$ of $M$.
  Furthermore, this pair of circuit and cocircuit signatures provides an orthogonal orientation of
  $M$. Any tame regular matroid is thus an orthogonally orientable matroid. In particular, the finite
  cycle matroid, the algebraic cycle matroid, and the topological cycle matroid of a graph are
  orthogonally orientable.
\end{exmp}

\subsection{Minors and Properties of Orthogonally Oriented Matroids}\label{ssec:MinorsAndPropsOrthOriMats}
From Definition~\ref{defn:OrthOrient} it follows directly that the dual matroid of every
orthogonally orientable matroid is an orthogonally orientable matroid as well. Another
consequence of this definition is that the circuit and cocircuit signatures belonging to an
orthogonally orientable matroid are passed on to minors in a unique way. This statement is
rendered more precisely in the following lemma.

\begin{lem}\label{lem:OrthImplUniSigInh}
  Let $(M, \CC, \CC^*)$ be an orthogonally oriented matroid on $E = X \dot{\cup} F \dot{\cup} G$
  according to Definition~\ref{defn:OrthOrient} and $N = M/F\!\setminus\!G$ be a minor of $M$.
  Then $\CC$ and $\CC^*$ induce a circuit signature and a cocircuit signature on $N$ with
  the following properties.
  \begin{enumerate*}
    \item If $\supp{C'}$ is an ordinary circuit of $N$ and $C$ is a signed circuit of $M$
      such that $\supp{C'} \subseteq \supp{C} \subseteq \supp{C'} \cup F$, then
      the signings $C'$ and $-C'$ of $\supp{C'}$ are given by $C' = C|_{\supp{C'}}$ and
      $-C' = -C|_{\supp{C'}}$.
    \item Dually, if $\supp{U'}$ is an ordinary cocircuit of $N$ and $U$ is a signed cocircuit
      of $M$ such that $\supp{U'} \subseteq \supp{U} \subseteq \supp{U'} \cup G$, then the
      signings $U'$ and $-U'$ of $\supp{U'}$ are given by $U' = U|_{\supp{U'}}$ and
      $-U' = -U|_{\supp{U'}}$.
  \end{enumerate*}
\end{lem}
\begin{proof}
  If $\supp{C'}$ is an ordinary circuit of $N$, then by \cite[Lemma 2.3]{exclminorinfmat} there
  exists a signed circuit $C$ of $M$ such that
  $\supp{C'} \subseteq \supp{C} \subseteq \supp{C'} \cup F$. Thus it is possible to
  orient $\supp{C'}$ as stated in (1). This signing of $\supp{C'}$ does not depend on
  the choice of $C$: Assume for a contradiction that there exists another signed circuit $D$ of $M$ such
  that $\supp{C'} \subseteq \supp{D} \subseteq \supp{C'} \cup F$ and
  $C|_{\supp{C'}} \ne D|_{\supp{C'}}, -D|_{\supp{C'}}$. Then
  there exist $e,f \in \supp{C'}$ such that $C(e) = -D(e)$ and $C(f) = D(f)$.
  Lemma~\ref{lem:2ElemCircCocircInters} lets us choose a cocircuit $\supp{U'}$
  of $N$ such that $\supp{C'} \cap \supp{U'} = \{e,f\}$. Again, by
  \cite[Lemma 2.3]{exclminorinfmat} there exists a signed cocircuit $U$ of $M$ such that
  $\supp{U'} \subseteq \supp{U} \subseteq \supp{U'} \cup G$. This implies
  $\supp{C'} \cap \supp{U'} = \supp{C} \cap \supp{U}
  = \supp{D} \cap \supp{U} = \{e,f\}$ which is not possible since both $C \perp U$ and
  $D \perp U$ hold. Thus we must have $C|_{\supp{C'}} = D|_{\supp{C'}}$ or $C|_{\supp{C'}}
  = -D|_{\supp{C'}}$.\\
  The dual statement (2) of (1) follows by applying the same arguments to the circuits of $N^*$.
\end{proof}

A direct result of Lemma~\ref{lem:OrthImplUniSigInh} is that the class of orthogonally
oriented matroids is closed under taking minors.

\begin{thm}\label{thm:OrthImplOrthMinors}
  Let $(M, \CC, \CC^*)$ be an orthogonally oriented matroid on $E = X \dot{\cup} F \dot{\cup} G$
  and $N = M/F\!\setminus\!G$ be a minor of $M$. Denote the circuit and cocircuit signatures
  induced by $\CC$ and $\CC^*$ on $N$ by $\CC_N$ and $\CC^*_N$, respectively. Then
  $(N, \CC_N, \CC^*_N)$ is an orthogonally oriented matroid on $X$, i.e.\ $N$ is an orthogonally
  orientable matroid.
\end{thm}
\begin{proof}
  Let $C'$ be a signed circuit of $N$ and $U'$ be a signed cocircuit of $N$. According to
  Lemmas~\ref{lem:StructCircOfMinors} and \ref{lem:OrthImplUniSigInh}, it is possible to
  choose a signed circuit $C$ and a signed cocircuit $U$ of $M$ such that
  $C' = C|_{\supp{C}}$ and $U' = U|_{\supp{U'}}$. It then follows from $\supp{C'} \cap \supp{U'}
  = \supp{C} \cap \supp{U}$ that $C' \perp U'$ holds.
\end{proof}

Given an orthogonally oriented matroid $\MC = (M, \CC, \CC^*)$ on a set $E$ and a subset
$X \subseteq E$, it is thus sensible to speak of the \textbf{orthogonally oriented restriction
minor $\bm{\MC|X}$}, the \textbf{orthogonally oriented deletion minor
$\bm{\MC\!\setminus\!\overline{X}}$}, and the \textbf{orthogonally oriented contraction minor
$\bm{\MC.X}$ or $\bm{\MC/\overline{X}}$ of $\bm{\MC}$}. For any minor $N$ of $M$, we denote the
circuit and cocircuit signatures induced by $\CC$ and $\CC^*$ on $N$ by $\CC_N$ and $\CC^*_N$,
respectively.\\

Alternatively, orthogonally oriented minors of orthogonally oriented matroids can be characterized
as follows.

\begin{cor}
  Let $\MC= (M, \CC, \CC^*)$ be an orthogonally oriented matroid on a set $E$ and $X \subseteq E$.
  \begin{enumerate*}
    \item $\CC_{M|X} = \CC_{M\!\setminus\!\overline{X}} = \{ C \in \CC\ |\ \supp{C} \subseteq X \}$
      is the circuit signature of an orthogonally oriented matroid on $X$. It is called the
      \textbf{orthogonally oriented restriction of $\bm{\MC}$ to $\bm{X}$}, denoted by $\MC|X$,
      or the \textbf{orthogonally oriented minor obtained by deleting $\bm{\overline{X}}$}, denoted
      by $\MC\!\setminus\!\overline{X}$.
    \item $\CC_{M.X} = \CC_{M/\overline{X}} = \min \{ C|_X\ |\ C \in \CC, \supp{C} \cap X \ne \emptyset \}$
      is a circuit signature  of an orthogonally oriented matroid on $X$. It is called the
      \textbf{orthogonally oriented contraction of $\bm{\MC}$ to $\bm{X}$}, denoted by $\MC.X$, or
      the \textbf{orthogonally oriented minor obtained by contracting $\bm{\overline{X}}$}, denoted
      by $\MC/\overline{X}$.
  \end{enumerate*}
\end{cor}
\begin{proof}
  The statements follow directly from the two previous Lemmas~\ref{lem:OrthImplUniSigInh} and
  \ref{thm:OrthImplOrthMinors}, combined with Proposition~\ref{prop:CharNOCMinor}.
\end{proof}

We now turn our attention to various other properties of orthogonally oriented matroids and
start by examining how circuit and cocircuit signatures of such matroids relate to each other.
Given an arbitrary circuit signature $\CC$ of a matroid $M$, it is not always possible to choose
a cocircuit signature $\CC^*$ such that the pair $\CC, \CC^*$ provides an orthogonal orientation
of $M$. However, if it is possible, then $\CC^*$ is uniquely determined, as the following
proposition shows.

\begin{prop}\label{prop:OrthAxiomsImplUniquenessOfSignature}
  Let $(M, \CC, \CC^*)$ be an orthogonally oriented matroid. If $\tilde{\CC}^*$ is another
  cocircuit signature of $M$ such that the pair $\CC, \tilde{\CC}^*$ satisfies (O) or (O'),
  then $\tilde{\CC}^* = \CC^*$.
\end{prop}
\begin{proof}
  Assume for a contradiction that there is a cocircuit $\supp{U}$ which is signed by $\CC^*$
  and $\tilde{\CC}^*$ in different ways. Denote the cocircuit signed by $\CC^*$ by $U$ and the
  cocircuit signed by $\tilde{\CC}^*$ by $\tilde{U}$. Then $U \ne \tilde{U}, -\tilde{U}$ and there
  exist elements $e, f \in \supp{U}$ such that $U(e) = \tilde{U}(e)$ and $U(f) = -\tilde{U}(f)$.
  By Lemma~\ref{lem:2ElemCircCocircInters} it is possible to pick a signed circuit $C$ such that
  $\supp{C} \cap \supp{U} = \{e,f\}$. Then either $C \perp U$ or $C \perp \tilde{U}$ holds,
  but not both, a contradiction.
\end{proof}

For finite oriented matroids, there is the well known concept of vectors and covectors. These are
to signed circuits and cocircuits of an oriented matroid what scrawls and coscrawls are to circuits
and cocircuits of an ordinary matroid. Indeed, the support of any vector/covector of an oriented
matroid is a scrawl/coscrawl of its underlying ordinary matroid, i.e.\ a union of circuits/cocircuits.
To take the signing of circuits and cocircuits into consideration, vectors and covectors are built by
composing circuits and cocircuits in an ordered fashion. In the infinite setting, it is thus natural
to utilize well-orderings for this purpose. Recall that a {\em well-ordering} $(S, <)$ on a set $S$
is a total order $<$ with the property that every non-empty subset of $S$ has a least element according
to this order. An {\em order isomorphism} between two well-orderings $(S, <), (S', <')$ is a
bijective map $i \colon S \rightarrow S'$ such that for all $s,t \in S$ it holds that $s < t$ if and
only if $i(s) <' i(t)$. A fundamental result states that every well-ordering is order isomorphic to a
unique ordinal number.\\

The definition of vectors and covectors rests on the more general definition of the composition of
signed subsets.

\begin{defn}\label{defn:InfSSComp}
  Let $\VC$ be a set of signed subsets.
  \begin{enumerate}
    \item Let $(\VC, <)$ be a well-ordering on $\VC$, let $\alpha$ be the ordinal number to which
      $(\VC, <)$ is order isomorphic, and denote the corresponding order isomorphism
      $\alpha \rightarrow \VC$ by $i$. Let $V^\beta = i(\beta)$ for all $\beta \in \alpha$. Then
      the signed subset $W$ with the properties
      \begin{enumerate*}
        \item[(i)] $\supp{W} = \bigcup_{V \in \VC} \supp{V}$
        \item[(ii)] $\forall e \in \supp{W}: \left(W(e) = V^\gamma(e),\ \textnormal{where}\ \gamma
          = \min\{ \beta \in \alpha\ |\ V^\beta(e) \ne 0 \}\right)$
      \end{enumerate*}
      is called the \textbf{composition of the signed subsets of $\bm{\VC}$ according to
      $\bm{(\VC, <)}$}.
      In this case we say that $\VC$ is the \textbf{decomposition of $\bm{W}$ according
      to $\bm{(\VC, <)}$}.
    \item A signed subset $W$ is called a \textbf{composition of the signed subsets of $\bm{\VC}$}
      if there exists a well-ordering $(\VC, <)$ on $\VC$ such that $W$ is the composition of the
      signed subsets of $\VC$ according to $(\VC, <)$. In this case we say that $\VC$ is a
      \textbf{decomposition of $\bm{W}$}.
  \end{enumerate}
\end{defn}

\begin{rem}
  For finite sets of signed subsets $\VC = \{ X_1, \ldots, X_n \}$, Definition~\ref{defn:InfSSComp}
  agrees with the usual definition of composition of signed subsets (see for
  instance \citet[Section 3.7]{orimatsbook}). In this case, it is customary to denote the composition
  of the signed subsets of $\VC$ by $X_1 \circ \ldots \circ X_n$.
\end{rem}

\begin{defn}\label{defn:VectorsOfMats}
  Let $M$ be a matroid that is equipped with a circuit signature. Then any composition of the
  signed circuits of $M$ is called a \textbf{vector (of $\bm{M}$)}. Dually, if $M$ is equipped
  with a cocircuit signature, then we call any composition of a set of signed cocircuits of $M$
  a \textbf{covector (of $\bm{M}$)}.
\end{defn}

\begin{defn}
  The \textbf{vectors and covectors of an orthogonally oriented matroid $\bm{(M, \CC, \CC^*)}$}
  are the vectors and covectors of $M$ according to $\CC$ and $\CC^*$.
\end{defn}

Vectors and covectors of an orthogonally oriented matroid preserve orthogonality; indeed, this
property is already inherent to the composition of signed subsets, as the following proposition
shows.

\begin{prop}\label{prop:OrthToAllElemThenOrthToComp}
  Let $\VC$ be a set of signed subsets. If $U$ is a signed subset such that $V \perp U$ for all
  $V \in \VC$, then $U$ is orthogonal to any composition of the signed subsets of $\VC$.
\end{prop}
\begin{proof}
  Let $(\VC, <)$ be a well-ordering on $\VC$ and let $\alpha$ be the ordinal number to which
  $(\VC, <)$ is order isomorphic. Denote the corresponding order isomorphism
  $\alpha \rightarrow \VC$ by $i$ and let $V^\beta := i(\beta)$ for all $\beta \in \alpha$. Let $W$
  be the composition of the signed subsets of $\VC$ according to $(\VC, <)$. If
  $\supp{W} \cap \supp{U} = \emptyset$, then the claim follows immediately. Otherwise,
  there is a minimal $\beta \in \alpha$ such that $\supp{V^\gamma} \cap U = \emptyset$ for
  all $\gamma < \beta$ and $\supp{V^\beta} \cap \supp{U} \ne \emptyset$. Thus there exist
  $e,f \in \supp{V^\beta}$ such that $V^\beta(e)U(e) = -V^\beta(f)U(f)$. Since $\beta$ is
  minimal, this implies $W(e) = V^\beta(e)$ and $W(f) = V^\beta(f)$, i.e.\ $W \perp U$.
\end{proof}

\begin{cor}\label{cor:OrthOfVecsAndCovecsOfOrthOriMat}
  Vectors and covectors of an orthogonally oriented matroid are orthogonal to each other.
\end{cor}
\begin{proof}
  Proposition~\ref{prop:OrthToAllElemThenOrthToComp}.
\end{proof}

\begin{rem}
  In the literature, vectors and covectors of finite oriented matroids are sometimes defined
  via orthogonality if the emphasis is not on composition and order. Any signed subset that
  is orthogonal to all signed cocircuits of an oriented matroid is then called a vector while
  covectors are defined vice versa. It is well known that this alternative definition and
  Definition~\ref{defn:VectorsOfMats} are equivalent for finite oriented matroids; see, for
  example, the end of Section~\ref{ssec:PPForSSS} below for a reason and a proof.
\end{rem}

We conclude this section by giving an excluded minors characterization of a certain class of
orthogonally orientable matroids: Let $k$ be a field. In \citep{exclminorinfmat}, Bowler and
Carmesin show that a tame matroid $M$ is $k$-representable or even regular if every finite
minor of $M$ is. The formulation and proof of this fact given in \citep{exclminorinfmat} serve
as a template for a similar result with respect to tame orthogonally orientable matroids and
yield the following theorem.

\begin{thm}\label{thm:OrthCharByMinors}
  Let $M$ be a tame matroid. Then the following statements are equivalent.
  \begin{enumerate*}
    \item $M$ is an orthogonally orientable matroid according to Definition~\ref{defn:OrthOrient}.
    \item Every finite minor of $M$ is an orthogonally orientable matroid (and thus the underlying
      matroid of an ordinary oriented matroid).
  \end{enumerate*}
\end{thm}
\begin{proof}
  It is an immediate consequence of Lemma~\ref{thm:OrthImplOrthMinors} that (1) implies (2). The
  reverse implication follows from a straightforward modification of the proof of
  \cite[Theorem 4.5, step (3) $\Rightarrow$ (1)]{exclminorinfmat}.
\end{proof}

\section{(No) Plain Circuit Axioms}\label{sec:NoPlainCircAxioms}
In the previous section, we have seen that transferring the orthogonality axioms for finite oriented
matroids to the infinite setting leads to a sensible class of (infinite) oriented matroids.
It is thus natural to ask if a similar approach is possible with respect to the circuit axioms
for finite oriented matroids. In the finite case, it is possible to translate the circuit axioms
for matroids to a set of circuit axioms for oriented matroids. It is even possible to require weak
signed circuit elimination only, since strong signed circuit elimination follows from weak
elimination in this case (see for instance \cite[Theorem 3.2.5]{orimatsbook}). Obviously, this weaker
prerequisite does not suffice in the infinite case since axiom (C3) of the circuit
axioms~\ref{defn:MatCircAxioms} already states that strong circuit elimination is a requirement for
infinite matroids. But does translating the circuit axioms for infinite matroids to the oriented
setting without adding further constraints eventually lead to another sensible class of (infinite)
oriented matroids? The answer is no and in this section, we examine why this is the case and why
axiom (C3) plays a crucial role in this context. In particular, we prove the following two
claims.\\

{\em Claim~1.} Strong signed circuit elimination still yields orthogonality in the infinite
case.\\

{\em Claim~2.} In general, as a property of infinite matroids, strong signed circuit
elimination is not inherited by contraction minors nor does it carry over to the dual of
a matroid.

\subsection{Strong Signed Circuit Elimination}\label{ssec:SSCircElim}
To render the two Claims~1 and 2 more precisely, we begin by translating the (strong) circuit
elimination property (C3) from Definition~\ref{defn:MatCircAxioms} to signed circuits as follows.

\begin{defn}\label{defn:StrongSignedCircuitElim}
  Let $M$ be a matroid and $\CC$ be a circuit signature of $M$. We say that $\CC$
  has the \textbf{strong circuit elimination property} if it satisfies the following condition:
  \begin{enumerate*}
    \item[(CE)] Whenever $C \in \CC, X \subseteq \supp{C}$, and $(C_x\ |\ x \in X)$ is a family
      of elements of $\CC$ such that $\supp{C_x} \cap X = \{x\}$ and $x \in \sep(C, C_x)$ for all
      $x \in X$, then for every $f \in \supp{C} \setminus \left( \bigcup_{x \in X} \sep(C, C_x) \right)$
      there exists a $D \in \CC$ such that
      $f \in \supp{D}$,
      $D^+ \subseteq \left( C^+ \cup \bigcup_{x \in X} C_x^+ \right) \setminus X$, and
      $D^- \subseteq \left( C^- \cup \bigcup_ {x \in X} C_x^- \right) \setminus X$.
  \end{enumerate*}
\end{defn}

\subsection{Strong Signed Circuit Elimination Yields Orthogonality}
We restate and prove Claim~1 as follows.

\begin{prop}\label{prop:SignedCircElimImplOrthAxioms}
  Let $M$ be a matroid and $\CC$ be a circuit signature of $M$ that has the strong circuit elimination
  property (CE). Then there exists a unique cocircuit signature $\CC^*$ of $M^*$
  such that the pair $\CC, \CC^*$ provides an orthogonal orientation of $M$.
\end{prop}
\begin{proof}
  Like in the finite case, we construct a cocircuit signature $\CC^*$ as follows. For each cocircuit
  $\supp{U}$ of $M$, fix an element $e_U \in \supp{U}$ and let $U(e_U) = 1$. Then,
  for each $e \in \supp{U} \setminus e_U$, pick a signed circuit $C_{U_e}$ with
  $\supp{C_{U_e}} \cap \supp{U} = \{e_U,e\}$. Let
  $U(e) = -\frac{C_{U_e}(e_U)}{C_{U_e}(e)}$. This implies $U \perp C_{U_e}$ for all
  $e \in \supp{U} \setminus e_U$.\\
  It remains to show that $U$ is orthogonal to every signed circuit. We prove this claim in two
  steps. First, assume for a contradiction that $C$ is a signed circuit containing $e_U$ such that $C$
  and $U$ are not orthogonal. Then the family of signed circuits
  $(C_{U_e}\ |\ e \in (\supp{C} \cap \supp{U}) \setminus e_U)$ has the property that
  $\supp{C_{U_e}} \cap \supp{C} \cap \supp{U} = \{e_U,e\}$. By replacing
  $C_{U_e}$ by $-C_{U_e}$ if necessary, we may assume that $C_{U_e}(e_U) = C(e_U)$ for all
  $e \in (\supp{C} \cap \supp{U}) \setminus e_U$. Since each of the signed circuits $C_{U_e}$
  is orthogonal to $U$, it follows that $e \in \sep(C, C_{U_e})$ for all
  $e \in (\supp{C} \cap \supp{U}) \setminus e_U$. It is thus possible to apply strong circuit
  elimination to $C, (C_{U_e}\ |\ e \in (\supp{C} \cap \supp{U}) \setminus e_U)$, and
  $e_U$ to obtain a signed circuit $C'$ such that $\supp{C'} \cap \supp{U} = \{e_U\}$, a
  contradiction. It follows that $C \perp U$ for all signed circuits $C$ such that
  $e_U \in \supp{C}$.\\
  To conclude the proof, assume for a contradiction that $C$ is a signed circuit that does not contain
  $e_U$ and is not orthogonal to $U$. By replacing $C$ by $-C$ if necessary, we may assume that
  $C(e) = U(e)$ for all $e \in \supp{C} \cap \supp{U}$. Pick an element
  $f \in \supp{C} \cap \supp{U}$ and consider the family of signed circuits
  $(C_{U_e}\ |\ e \in (\supp{C} \cap \supp{U}) \setminus f)$. By replacing $C_{U_e}$
  by $-C_{U_e}$, we may assume that $e \in \sep(C, C_{U_e})$ for all
  $e \in (\supp{C} \cap \supp{U}) \setminus f$. Since $C_{U_e}(e) = -C(e) = -U(e)$
  and $C_{U_e} \perp U$, it must hold that $U(e_U) = C_{U_e}(e_U)$ for all
  $e \in (\supp{C} \cap \supp{U}) \setminus f$. We apply strong circuit elimination to
  $C, (C_{U_e}\ |\ e \in (\supp{C} \cap \supp{U}) \setminus f)$, and $f$ to obtain a signed
  circuit $C'$ such that $\supp{C'} \cap \supp{U} = \{e_U,f\}, C'(e_U) = C_{U_e}(e_U)
  = U(e_U)$, and $C'(f) = C(f) = U(f)$. This is a contradiction since $e_U \in \supp{C'}$ and
  thus $C' \perp U$ must hold by the first part of the proof. It follows that $C \perp U$ for all
  signed circuits $C$.\\
  Finally, the uniqueness of the cocircuit signature follows from
  Proposition~\ref{prop:OrthAxiomsImplUniquenessOfSignature}.
\end{proof}

\begin{rem}
  Note that the proof of Proposition~\ref{prop:SignedCircElimImplOrthAxioms} is constructive: it
  explains how a suitable cocircuit signature $\CC^*$ of $M$ can be constructed when given a
  circuit signature of $M$ that has the strong circuit elimination property (CE).
\end{rem}

Before we substantiate Claim~2 and prove its negative outcome, we briefly state two positive
results that are straightforward extensions of the finite case.

\begin{prop}
  Let $M$ be a matroid on $E$, let $\CC$ be a circuit signature of $M$, and let $G \subseteq E$.
  If $\CC$ has the strong circuit elimination property (CE), then the circuit signature induced
  by $\CC$ on $M|G$ has this property as well.
\end{prop}
\begin{proof}
  This is an immediate consequence of the definition of restriction minors.
\end{proof}

\begin{prop}
  Let $M$ be a matroid on $E$ and $\CC$ be a circuit signature of $M$ that has the strong elimination
  property (CE). Let $\emptyset \ne F \subseteq E$ and $C, D \in \CC$ such that
  $\supp{C|_F}$ and $\supp{D|_F}$ are minimal and non-empty. If
  $\supp{C|_F} = \supp{D|_F}$, then $C|_F = D|_F$ or $C|_F = -D|_F$.
\end{prop}
\begin{proof}
  Assume $C|_F \ne D|_F$ and $C|_F \ne -D|_F$. Then it is possible to choose $e \in \sep(C|_F, D|_F)$
  and $f \in (C|_F^+ \cap D|_F^+) \cup (C|_F^- \cap D|_F^-)$. By applying strong circuit
  elimination to $C, D, e$, and $f$, we obtain a circuit $C'$ such that
  $\emptyset \ne \supp{C'|_F} \subset \supp{C|_F}$, a contradiction.
\end{proof}

\subsection{Strong Signed Circuit Elimination in Relation to Contraction Minors and Duals}
We now enter the core of this section and show that there is a matroid $M$ and a circuit
signature $\CC$ of $M$ such that $M$ and $\CC$ exhibit the characteristics stated in Claim~2.
A suitable choice for $M$ must be an orthogonally orientable matroid and $\CC$ must possess
the strong circuit elimination property (CE). The construction of such an $M$ and $\CC$ is
based on the following two definitions.

\begin{defn}
  We call a point on the unit sphere $S^2$ in $\RR^3$ \textbf{lexicographically positive} if its
  first non-zero component is positive and \textbf{lexicographically negative} otherwise.
\end{defn}

\begin{defn}
  Let $Q$ be a set of lines through the origin in $\RR^3$.
  \begin{enumerate*}
    \item We say that $Q$ is \textbf{free} if no three lines of $Q$ are coplanar.
    \item We say that $Q$ is \textbf{$\bm{S^2}$-dense} if for any $S^2$-neighborhood $X$ it
      contains a line that meets $X$.
    \item We say that $Q$ is \textbf{neat} if it is free and for any four lines $a,b,c,d \in Q$
      with $a$ being distinct from $b$ and $c$ being distinct from $d$ the following condition
      is satisfied: Let $e \in Q \setminus \{a,b,c,d\}$ and denote the intersection of the planes
      spanned by $ab$ and $cd$ by $l$. If $l$ is a line (through the origin), then there is a
      further line $f \in Q \setminus \{e\}$ such that the planes spanned by $ab,cd$, and $ef$
      meet in $l$. See Figures~\ref{fig:NeatSetFiveLines} and \ref{fig:NeatSetSixLines} for an
      illustration of this property.
    \item We say that $Q$ is \textbf{messy} if it is free, $S^2$-dense, and contains no six
      distinct lines $a,b,c,d,e,f$ such that the intersection of the planes spanned by $ab,cd$,
      and $ef$ is a line.
  \end{enumerate*}
\end{defn}

\begin{figure}[h]
  \centering
  \includegraphics[width=0.95\textwidth]{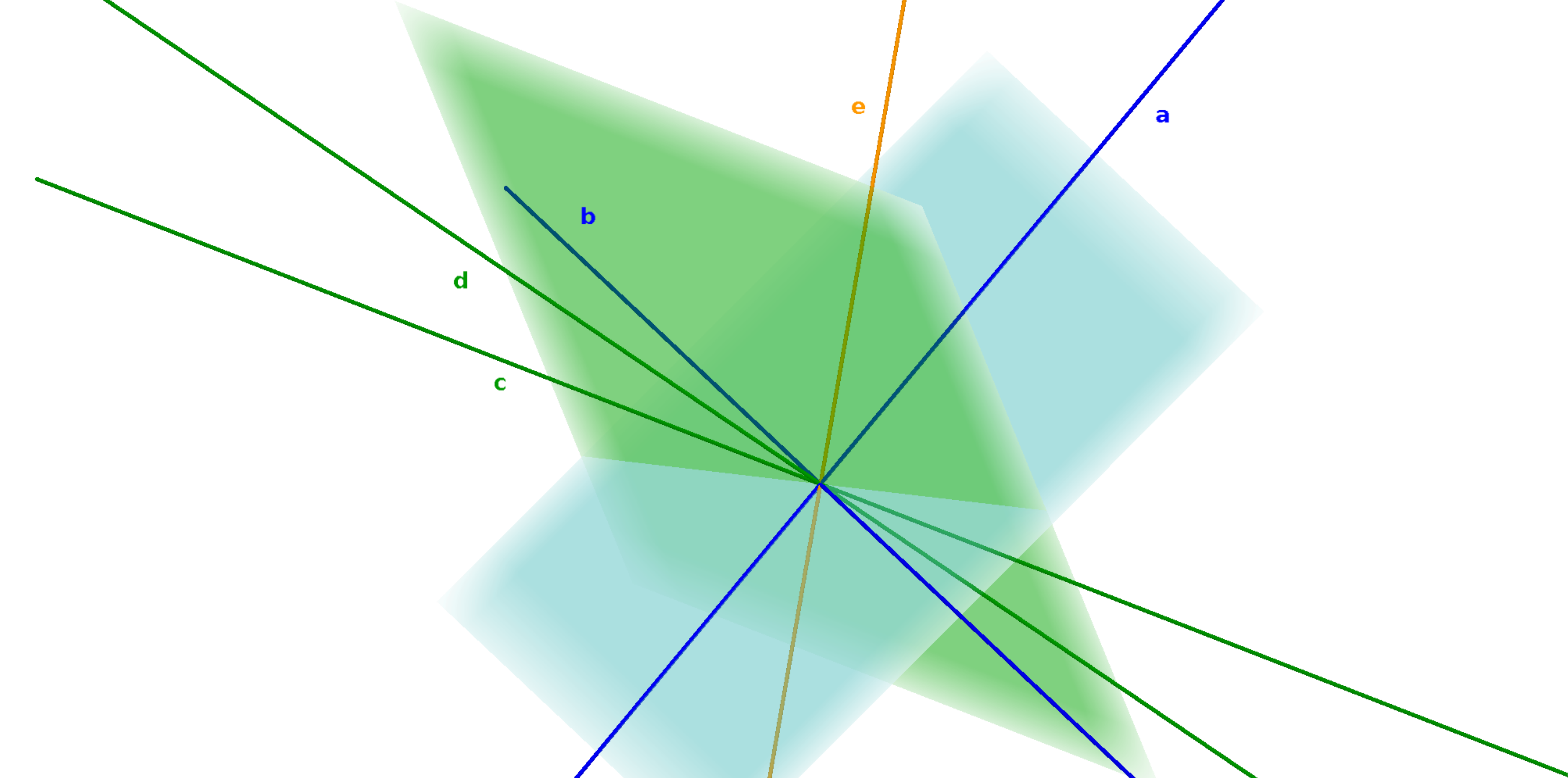}
  \caption{Five lines $a,b,c,d,e$ through the origin in $\RR^3$ such that the intersection
           of the planes spanned by $ab$ (blue) and $cd$ (green) is a line through the origin}
  \label{fig:NeatSetFiveLines}
\end{figure}

\begin{figure}[ht]
  \centering
  \includegraphics[width=0.95\textwidth]{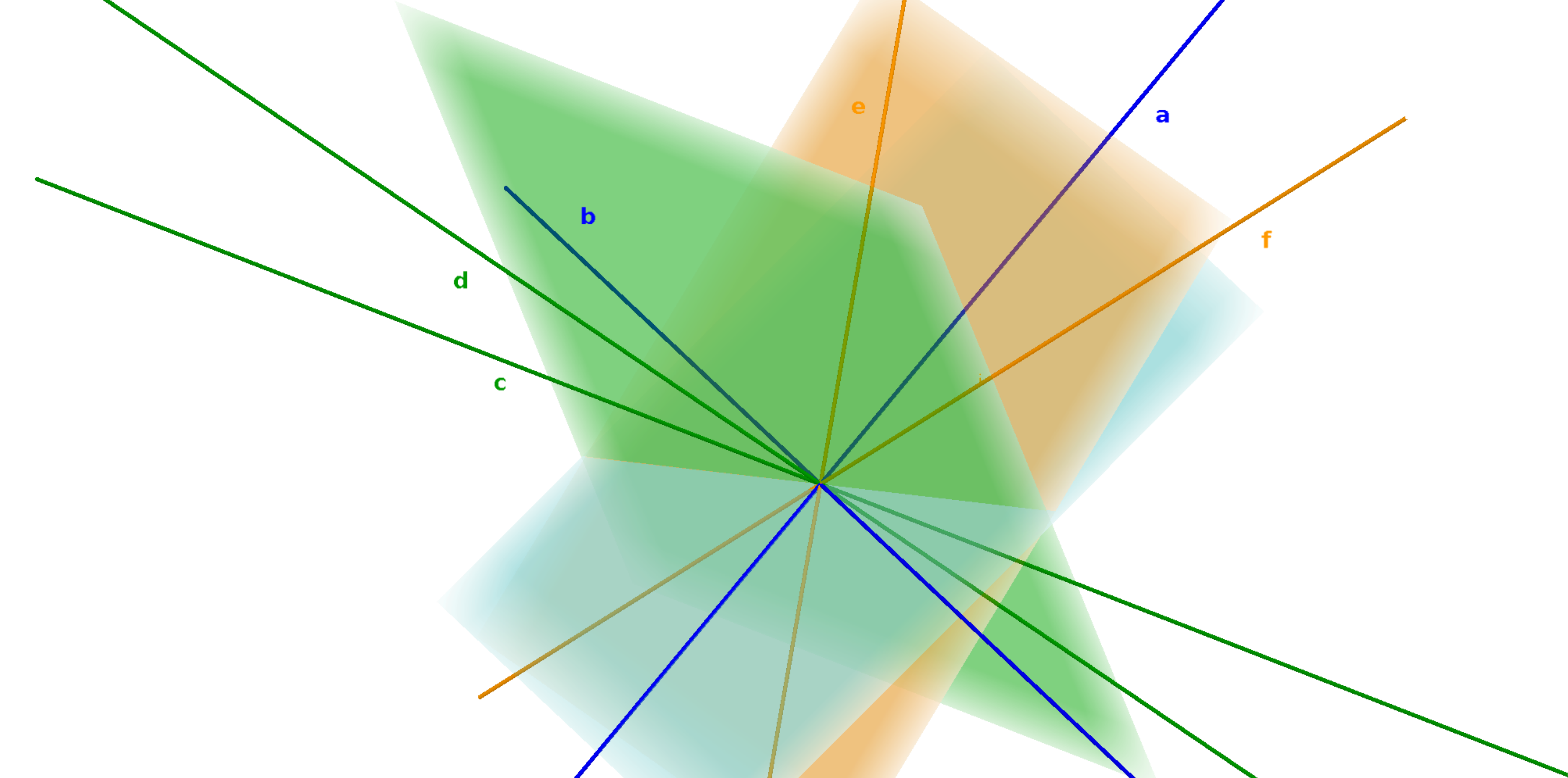}
  \caption{If the lines $a,b,c,d,e$ from Figure~\ref{fig:NeatSetFiveLines} belong to a neat
           set of lines, then there exists a further line $f$ such that the intersection of
           the planes spanned by $ab$ (blue), $cd$ (green), and $ef$ (orange) is a line again}
  \label{fig:NeatSetSixLines}
\end{figure}

Now, let $Q$ be a free set of lines through the origin in $\RR^3$ and consider the uniform matroid
$\UC_{3,Q}$. This matroid is equipped with a cocircuit signature as follows: Let
$Q \setminus \{a,b\}$ be a cocircuit of $\UC_{3,Q}$. The intersection of the complement of the plane
spanned by $ab$ with $S^2$ consists of two disjoint components of $S^2$. An opposing sign is then
assigned to each of these. Furthermore, the intersection of any line $l \in Q$ with $S^2$
yields a pair of antipodal points on $S^2$ where one of the points is lexicographically positive
whereas the other one is lexicographically negative. Since $Q$ is free, it is thus possible to orient
a cocircuit as follows. Every element of $Q \setminus \{a,b\}$ is assigned the sign of the component
of $S^2$ in which its corresponding lexicographically positive point on $S^2$ is located. Denote the
signed cocircuit obtained in this way by $Q_{a,b}$. Repeating the same construction considering only
lexicographically negative points then yields the signed cocircuit which is equal to $-Q_{a,b}$. The
collection of these signed subsets is thus a cocircuit signature of $\UC_{3,Q}$, and we denote
this cocircuit signature by $\CC^*_{\UC_{3,Q}}$ for the remainder of this section.\\

The following lemma and corollary show that $\CC^*_{\UC_{3,Q}}$ has the strong circuit elimination
property (CE) if $Q$ is neat. By the dual statement of
Proposition~\ref{prop:SignedCircElimImplOrthAxioms}, both $\UC_{3,Q}$ and its dual are thus
orthogonally orientable on $Q$.

\begin{lem}\label{lem:CircElimAndNeatSets}
  If $Q$ is neat, then $\CC^*_{\UC_{3,Q}}$ fulfills the following weak type of circuit elimination:
  Let $U_1,U_2 \in \CC^*_{\UC_{3,Q}}$, let $e \in \sep(U_1,U_2)$, and let
  $z \in \supp{U_1} \setminus \sep(U_1,U_2)$. Then there exists a $U \in \CC^*_{\UC_{3,Q}}$
  such that $z \in \supp{U}, U^+ \subseteq (U_1^+ \cup U_2^+) \setminus e$, and
  $U^- \subseteq (U_1^- \cup U_2^-) \setminus e$.
\end{lem}
\begin{proof}
  Let $\supp{U_1} = Q \setminus \{a,b\}$ and $\supp{U_2} = Q \setminus \{c,d\}$. Since $Q$ is free,
  the intersection of the planes spanned by $a,b$ and $c,d$ is a line. Thus there exists a further
  line $f \in Q \setminus e$ such that the planes spanned by $ab, cd$, and $ef$ meet in a line
  through the origin. Note that there is a signed cocircuit with support $Q \setminus \{e,f\}$,
  i.e it is sensible to talk about signed components after removing the plane spanned by $ef$ in the
  following. Let $R_{ab}^+$ resp. $R_{ab}^-$ be the component of $S^2$ which is signed positively
  resp. negatively with respect to the complement of the plane spanned by $ab$. Similarly, define
  $R_{cd}^+, R_{cd}^-$ resp. $R_{ef}^+, R_{ef}^-$ with respect to the complement of the plane
  spanned by $cd$ resp. $ef$. It then suffices to examine the following two of four possible cases.

  \begin{enumerate}
    \item The signings of both $U_1$ and $U_2$ are obtained by considering lexicographically
      positive points of $S^2$ only.
    \item The signing of $U_1$ is obtained by considering lexicographically positive points
      of $S^2$ whereas the signing of $U_2$ is obtained by considering lexicographically
      negative points of $S^2$.
  \end{enumerate}

  In case (1), one of the two components $R_{ef}^+$ and $R_{ef}^-$ includes $R_{ab}^+ \cap R_{cd}^+$
  while the other one includes $R_{ab}^- \cap R_{cd}^-$. There are two possible signings of
  the cocircuit with support $Q \setminus \{e,f\}$. Let $U$ be the one that satisfies
  $(U_1^+ \cap U_2^+) \setminus f \subseteq U^+$ and $(U_1^- \cap U_2^-) \setminus f \subseteq U^-$.
  We claim that $U^+ \subseteq (U_1^+ \cup U_2^+) \setminus e$ and
  $U^- \subseteq (U_1^- \cup U_2^-) \setminus e$ holds. Let $g \in U^+$. We only have to consider
  the case $g \in \{a,b,c,d\}$ since otherwise $g \in U_1^+ \cup U_2^+$ follows directly. Now, two
  subcases are possible, namely that $a,b,c,d$ are distinct or not. If $a,b,c,d$ are distinct, then
  the component $R_{ef}^+$ or $R_{ef}^-$ which includes $R_{ab}^+ \cap R_{cd}^+$ also includes
  the intersections of the lines $a,b,c,d$ with $S^2$ since $Q$ is free. Thus either $U(g) = U_1(g)$
  or $U(g) = U_2(g)$ holds, i.e.\ $g \in U_1^+ \cup U_2^+$. If $a,b,c,d$ are not distinct, then either
  $a \in \{c,d\}$ or $b \in \{c,d\}$. Without loss of generality we may assume $a = c$. Since $Q$ is
  free, $a = c = f$ must hold in this case and by consequence $g \ne a,c$. By the same argument as in
  the case that $a,b,c,d$ are distinct we conclude that either $U(g) = U_1(g)$ or $U(g) = U_2(g)$,
  i.e.\ $g \in U_1^+ \cup U_2^+$ again. Thus $U^+ \subseteq (U_1^+ \cup U_2^+) \setminus e$. Similarly,
  we deduce that $U^- \subseteq (U_1^- \cup U_2^-) \setminus e$ holds. Finally, since $Q$ is free either
  $f \in \sep(U_1,U_2)$ or $f \notin \supp{U_1}$, i.e.\ $z \in \supp{U}$. This concludes the proof for
  case (1).

  The proof for case (2), is similar except that one of the two components $R_{ef}^+$ and $R_{ef}^-$
  includes $R_{ab}^+ \cap R_{cd}^-$ while the other one includes $R_{ab}^- \cap R_{cd}^+$.
\end{proof}

\begin{cor}\label{cor:NeatImplCE}
  If $Q$ is neat, then $\CC^*_{\UC_{3,Q}}$ has the strong circuit elimination property (CE).
\end{cor}
\begin{proof}
  Let $U \in \CC^*_{\UC_{3,Q}}, X \subseteq \supp{U}, (U_x\ |\ x \in X)$, and
  $f \in \supp{U} \setminus \left( \bigcup_{x \in X} \sep(U, U_x) \right)$ be as stated in the
  definition of (CE). Then for any $x \in X$, we have $\supp{U_x} \cap X = \{x\}$, so that
  $X \subseteq (Q \setminus \supp{U_x}) \cup x$, i.e.\ $X$ has at most three elements. Thus we
  have to consider the following three cases only.

  \begin{enumerate}
    \item $X = \{ x \}$: One application of Lemma~\ref{lem:CircElimAndNeatSets} yields the
      desired result.
    \item $X = \{ x_1, x_2 \}$: Applying Lemma~\ref{lem:CircElimAndNeatSets} to $U$ and
      $U_{x_1}$ yields a cocircuit $V$ such that
      $V^+ \subseteq \left( U^+ \cup U^+_{x_1} \right) \setminus x_1$,
      $V^- \subseteq \left( U^- \cup U^-_{x_1} \right) \setminus x_1$, and $f \in \supp{V}$.
      If $x_2 \notin \supp{V}$, then $V$ already has all of the desired properties. Otherwise
      note that since $V(f) = U(f)$ and $f \notin \sep(U,U_{x_2})$ also
      $f \notin \sep(V,U_{x_2})$. On the other hand, since $x_2 \notin \supp{U_{x_1}}$ we
      must have $V(x_2) = U(x_2)$, i.e.\ $x_2 \in \sep(V,U_{x_2})$. It is thus possible to
      apply Lemma~\ref{lem:CircElimAndNeatSets} to $V$ and $U_{x_2}$ to obtain the desired
      cocircuit.
    \item $X = \{ x_1, x_2, x_3 \}$: If $X$ contains exactly three elements, it is not possible
      to choose an $f$ as required in the definition of (CE),
      i.e.\ $\supp{U} \setminus \left(\bigcup_{x \in X} \sep(U,U_x)\right)$ must be the empty set.
      Otherwise the reasoning from case (2) would imply that it is possible to apply
      Lemma~\ref{lem:CircElimAndNeatSets} three times to obtain a cocircuit whose support
      excludes three elements of $Q$, a contradiction.
  \end{enumerate}
\end{proof}

On the contrary, if $P$ is a messy set of lines through the origin in $\RR^3$, then
$\CC^*_{\UC_{3,P}}$ does not possess the strong circuit elimination property (CE).

\begin{lem}
  If $P$ is a messy set of lines through the origin in $\RR^3$, then strong circuit elimination (CE)
  fails for $\CC^*_{\UC_{3,P}}$ in the following strong sense: for any signed cocircuits
  $U_1$ and $U_2$ in $\CC^*_{\UC_{3,P}}$ with $\supp{U_1} \cup \supp{U_2} = P$ and any
  $e \in \sep(U_1,U_2)$, there is no $U \in \CC^*_{\UC_{3,P}}$ with
  $U^+ \subseteq (U_1^+ \cup U_2^+) \setminus e$ and $U^- \subseteq (U_1^- \cup U_2^-) \setminus e$.
\end{lem}
\begin{proof}
  Suppose for a contradiction that there are such $U_1, U_2, e$, and $U$. Let
  $\supp{U_1} = P \setminus \{a,b\}, \supp{U_2} = P \setminus \{c,d\}$, and
  $\supp{U} = P \setminus \{e,f\}$. Define $R_{ab}^+, R_{ab}^-, R_{cd}^+, R_{cd}^-, R_{ef}^+$,
  and $R_{ef}^-$ as in the proof of Lemma~\ref{lem:CircElimAndNeatSets}. Since no three of the
  lines $a,b,c,d$ are coplanar and $P$ is messy, the intersection of the planes spanned by
  $ab, cd$, and $ef$ is not a line through the origin. Considering the two cases from
  Lemma~\ref{lem:CircElimAndNeatSets} again, this means that the plane spanned by $ef$ cuts
  the regions $R_{ab}^+ \cap R_{cd}^+$ and $R_{ab}^- \cap R_{cd}^-$ in case (1) and the regions
  $R_{ab}^+ \cap R_{cd}^-$ and $R_{ab}^- \cap R_{cd}^+$ in case (2), respectively. In either
  case, since $P$ is $S^2$-dense, it follows that
  $U^+ \cap (U_1^- \cup U_2^-) \setminus \sep(U_1,U_2) \ne \emptyset$ or
  $U^- \cap (U_1^+ \cup U_2^+) \setminus \sep(U_1,U_2) \ne \emptyset$, a contradiction.
\end{proof}

The next step is to show that there actually exists a suitable neat set of lines through the
origin that has a messy subset.

\begin{lem}\label{lem:NeatQMessyP}
  There is a neat set $Q$ of lines through the origin in $\RR^3$, and a further subset $P$
  with $P \subseteq Q$ which is messy.
\end{lem}
\begin{proof}
  Let $(O_i\ |\ i \in \NN_0)$ be a countable basis of the topology on $S^2$. Consider the set
  of 5-element lists
  \[
    \left\{ (p_1,p_2,p_3,p_4,p_5) \in \ZZ^5\ \middle|\ p_1 \ne p_2, p_3 \ne p_4,
                                                       \left| \{ p_1, p_2 \} \cap \{ p_3, p_4 \} \right| \leq 1,
                                                       p_5 \notin \{ p_1, p_2, p_3, p_4 \} \right\}
  \]
  and let $(p^n\ |\ n \in \NN_0)$ be an enumeration of this set chosen in such a way that for
  any $n \in \NN_0$ all $p^n_t$ are less than $n$.

  First, we build a sequence $(x_i\ |\ i \in \NN_0)$ recursively with $x_i \in O_i$. For all
  $n \in \NN_0$ let $\tilde{l}_i$ be the line through $x_i$ and the origin in $\RR^3$. To ensure
  that the set $P = \{ \tilde{l}_i\ |\ i \in \NN_0 \}$ is messy, we choose $x_i$ so that it is
  distinct from all previous $x_j$, does not lie on the great circle through $x_j, x_k$ for
  any $x_j, x_k$ with $j,k$ distinct, and does not lie on any of the finitely many great circles
  joining some $x_j$ to the intersections of some great circle through $x_k, x_l$ with some
  great circle through $x_m, x_n$, all of $j,k,l,m$, and $n$ being distinct and less than $i$.
  Since we only have to avoid finitely many points and great circles, this is always possible.\\
  Now we define a line $l_i$ through the origin in $\RR^3$ for each $i \in \ZZ$ as follows.
  First, for $i < 0$, we take $l_i = \tilde{l}_{-1-i}$. Then we define the $l_i$ with $i \in \NN_0$
  recursively. If there is no integer $j < i$ for which the planes spanned by
  $l_{p_1^i} l_{p_2^i}, l_{p_3^i} l_{p_4^i}$, and $l_{p_5^i} l_j$ meet in a line through the
  origin, then we take $l_i$ to lie in the plane $h_i$ spanned by $l_{p_5^i}$ and the intersection
  of the planes spanned by $l_{p_1^i} l_{p_2^i}$ and $l_{p_3^i} l_{p_4^i}$. Otherwise, we pick any
  plane $h_i$ through the origin including at most one $l_j$ with $j < i$ and take $l_i$ to lie
  in this plane. To ensure that the set $Q = \{ l_i\ |\ i \in \ZZ \}$ is free, we choose $l_i$ in
  such a way that it lies in none of the planes spanned by any $l_j l_k$ with $j$ and $k$ integers
  less than $i$. This is possible since there are only countably many such planes, and each of them
  meets the plane $h_i$ in at most one line.
\end{proof}

We are almost ready to present the desired example. The final ingredient is the following
lemma which is used to validate the second statement of Claim~2.

\begin{lem}\label{lem:FOrthOMImplSignedCircElim}
  The circuit signature of a finitary orthogonally oriented matroid has the strong circuit
  elimination property (CE).
\end{lem}
\begin{proof}
  Consider a finitary orthogonally oriented matroid with underlying matroid $M$ and circuit
  signature $\CC$. First, let $C, C' \in \CC$ be such that $C \ne -C'$ and
  $\sep(C, C') \ne \emptyset$ and let $f \in \supp{C} \setminus \sep(C, C')$ and
  $e \in \sep(C, C')$. By Lemma~\ref{thm:OrthImplOrthMinors} the finite restriction minor
  $N = M|(\supp{C} \cup \supp{C'})$ of $M$ is the underlying matroid of a finite oriented
  matroid whose circuit signature $\CC_N$ is induced by $\CC$. Thus there exists a
  $D \in \CC_N \subseteq \CC$ such that $f \in \supp{D},
  D^+ \subseteq (C^+ \cup C'^+) \setminus e$, and $D^- \subseteq (C^- \cup C'^-) \setminus e$. Now,
  let $C \in \CC, X \subseteq \supp{C}, (C_x\ |\ x \in X)$ be a family of elements of $\CC$
  such that $\supp{C_x} \cap X = \{x\}$ and $x \in \sep(C, C_x)$ for all $x \in X$, and let
  $f \in \supp{C} \setminus \left( \bigcup_{x \in X} \sep(C, C_x) \right)$. Note that
  $(C_x\ |\ x \in X)$ is a finite family since $\supp{C}$ is finite. Let
  $X = \{x_1, \ldots, x_n\}$ and set $C_0 = C$. The two conditions $x \in \sep(C, C_x)$ and
  $\supp{C_x} \cap X = \{x\}$ for all $x \in X$ and the first part of the proof imply that
  there exists a chain of
  circuits $C_0, C_1, \ldots, C_n \in \CC$ such that $f \in \supp{C_i}$ for
  $0 \leq i \leq n$ and additionally
  $C_j^+ \subseteq (C_{j-1}^+ \cup C_{x_j}^+) \setminus \{x_1, \ldots, x_j\}$ and
  $C_j^- \subseteq (C_{j-1}^- \cup C_{x_j}^-) \setminus \{x_1, \ldots, x_j\}$ for
  $1 \leq j \leq n$. The circuit $C_n$ then satisfies $f \in \supp{C_n},
  C_n^+ \subseteq \left( C^+ \cup \bigcup_{x \in X} C_x^+ \right) \setminus X$, and
  $C_n^- \subseteq \left( C^- \cup \bigcup_ {x \in X} C_x^- \right) \setminus X$.
\end{proof}

\begin{exmp}\label{exmp:OrthOriMAtWithSSCEAndNonSSCECMinor}
  Choose $Q$ and $P$ according to Lemma~\ref{lem:NeatQMessyP}. Then $\UC_{3,Q}$ is the underlying
  matroid of an orthogonally oriented matroid on $Q$ by Corollary~\ref{cor:NeatImplCE}. The
  cocircuit signature $\CC^*_{\UC_{3,Q}}$ of this orthogonally oriented matroid has the strong circuit
  elimination property (CE). However, the cocircuit signature $\CC^*_{\UC_{3,P}}$ induced by
  $\CC^*_{\UC_{3,Q}}$ on the contraction minor $\UC_{3,P}$ of $\UC_{3,Q}$ no longer has this property.
  Furthermore, by Lemma~\ref{lem:FOrthOMImplSignedCircElim} the orthogonally oriented matroid
  $(\UC_{3,P}, \CC_{\UC_{3,P}}, \CC^*_{\UC_{3,P}})$ provides an example of a finitary orthogonally
  orientable matroid that is equipped with a circuit and cocircuit signature such that the circuit
  signature has the strong circuit elimination property (CE) whereas the cocircuit signature does not
  have this property.
\end{exmp}

Example~\ref{exmp:OrthOriMAtWithSSCEAndNonSSCECMinor} shows that both statements of Claim~2 are true.
Just translating the circuit axioms \ref{defn:MatCircAxioms} to the oriented setting as in the
finite case and replacing axiom (C3) by (CE) thus does not lead to another sensible class of
infinite oriented matroids since contraction minors in general do not satisfy (CE) again. Note
that this does not mean that it is impossible to state some kind of circuit axioms for infinite
oriented matroids; it just means that a plain translation of the circuit
axioms \ref{defn:MatCircAxioms} does not succeed.

\begin{rem}
  Note that Example~\ref{exmp:OrthOriMAtWithSSCEAndNonSSCECMinor} is a counterexample to
  \citep[Conjecture 3.13]{orthaxiomsinforimats}.
\end{rem}

Regardless of the negative implications of Example~\ref{exmp:OrthOriMAtWithSSCEAndNonSSCECMinor},
Proposition~\ref{prop:SignedCircElimImplOrthAxioms} shows that even in the infinite case, it is still
possible to show that a matroid is orthogonally orientable by choosing a circuit (or cocircuit)
signature and proving that it has the strong circuit elimination property (CE). In this spirit, we
conclude this section with the following observation.

\begin{cor}\label{cor:CCircAxiomOMImplCircAxiomOnDual}
  Let $M$ be a cofinitary matroid and $C$ be a circuit signature of $M$ that has the strong circuit
  elimination property (CE). Then there exists a unique cocircuit signature $C^*$ of $M$
  such that the pair $\CC, \CC^*$ provides an orthogonal orientation of $M$. In particular,
  $C^*$ has the strong elimination property (CE) in this case as well.
\end{cor}
\begin{proof}
  Proposition~\ref{prop:SignedCircElimImplOrthAxioms} and Lemma~\ref{lem:FOrthOMImplSignedCircElim}.
\end{proof}

\begin{exmp}\label{exmp:CMWithoutFarkas}
  Let $E = \NN$ and $M$ be the uniform infinite matroid of rank 2 on $E$. The circuits of this
  finitary matroid are the subsets of $\NN$ that contain exactly three elements. Its cocircuits
  are the sets $\NN \setminus \{i\}$ where $i \in \NN$. We introduce a circuit signature $\CC$
  and a cocircuit signature $\CC^*$ of $M$ as follows. First, we want to assign to each circuit
  of $M$ the alternating patterns $(+,-,+)$ and $(-,+,-)$. Formally: We consider the circuits of $M$
  to be triples
  $(i,j,k)$ where $i,j,k \in \NN$ and $1 \leq i < j < k$. Then, we denote by $C_{i,j,k} \in \TPME$
  the signed subset that satisfies $C_{i,j,k}^+ = \{i,k\}$ and $C_{i,j,k}^- = \{j\}$, and set
  $\CC = \{ C_{i,j,k}\ |\ i,j,k \in \NN, 1 \leq i < j < k \} \cup
  \{ -C_{i,j,k}\ |\ i,j,k \in \NN, 1 \leq i < j < k \}$. A cocircuit $\NN \setminus \{i\}$ of $M$
  is signed by assigning a different sign to each of its "legs" $\{ j \in \NN\ |\ j < i \}$
  and $\{ j \in \NN\ |\ j > i \}$: denote by $U_i \in \TPME$ the signed subset of $N$ that
  satisfies $U_i^+ = \{ j \in \NN\ |\ j < i \}$ and $U_i^- = \{ j \in \NN\ |\ j > i \}$, and set
  $\CC^* = \{ U_i\ |\ i \in \NN \} \cup \{ -U_i\ |\ i \in \NN \}$.

  \begin{figure}[h]
    \[
    \begin{array}{rllllllllll}
      E & 1 & 2 & 3 & 4 & 5 & 6 & 7 & 8 & 9 & \ldots\\
      C_{1,2,3} & + & - & + &   &   &   &   &   &   &\\
      C_{1,3,4} & + &   & - & + &   &   &   &   &   &\\
      -C_{4,6,8} &   &   &   & - &   & + &   & - &   &\\
      U_1 &   & - & - & - & - & - & - & - & - & \ldots\\
      -U_5 & - & - & - & - &   & + & + & + & + & \ldots
    \end{array}
    \]
    \caption{Signed circuits and cocircuits of $M$ from Example~\ref{exmp:CMWithoutFarkas}}
  \end{figure}

  We claim that both $\CC$ and $\CC^*$ have the strong circuit elimination property (CE).
  To see that this is true, note that it is not difficult to check that $\CC^*$ has the property
  (CE) by a simple case distinction. The claim then follows from the dual statement of
  Corollary~\ref{cor:CCircAxiomOMImplCircAxiomOnDual}; in particular, the pair $\CC, \CC^*$
  provides an orthogonal orientation of $M$.
\end{exmp}

\subsection{Orthogonality Allows Weaker Types of Circuit Elimination}
In the previous section, we have observed that---in contrast to the finite case---orthogonality
is a rather weak property of generic infinite matroids: while finite orthogonally oriented
matroids (i.e.\ oriented matroids in the usual sense) permit strong circuit elimination, this is
in general not the case for infinite orthogonally oriented matroids. Due to this fact, it is
worthwhile to note that some properties of finite orthogonally oriented matroids related to
ordinary circuit elimination that do not seem to be worth stating separately in the finite case
nevertheless carry over to the infinite setting. In this section, we exhibit two examples of
such properties.\\

Recall that every orthogonally oriented matroid is based on an underlying ordinary matroid.
These matroids form a special subclass of the class of ordinary matroids. Orthogonality adds
structural information and this additional information can be used to refine the circuit
axiom (C3) about ordinary strong circuit elimination as follows.

\begin{prop}\label{prop:OrthImplSpecialCircElim}
  Let $(M, \CC, \CC^*)$ be an orthogonally oriented matroid, $C \in \CC,
  X \subseteq \supp{C}$, and $(C_x\ |\ x \in X)$ be a family
  of elements of $\CC$ such that $\supp{C_x} \cap X = \{x\}$ and
  $x \in \sep(C, C_x)$ for all $x \in X$. Then for every
  $f \in \supp{C} \setminus \left( \bigcup_{x \in X} \sep(C, C_x) \right)$ there exists
  a $D \in \CC$ such that $f \in \supp{D}, D(f) = C(f)$, and $\supp{D}
  \subseteq \left( \supp{C} \cup \bigcup_{x \in X} \supp{C_x} \right) \setminus X$.
\end{prop}
\begin{proof}
  Let $G = \left( \supp{C} \cup \bigcup_{x \in X} \supp{C_x} \right)$. Since circuits of
  the minor $M|G$ are also circuits of $M$, it suffices to show that there exists a
  circuit $\supp{D} \in \CC(M|G)$ that has the desired properties. To do so, we
  first prove that $X \cup f$ is coindependent in $M|G$. Assume for a contradiction
  that $X \cup f$ contains a cocircuit $\supp{U}$ of $M|G$. Then one of the following
  cases must hold.
  \begin{enumerate*}
    \item $\supp{U} = \{f\}$: this case implies the contradiction
      $|\supp{C} \cap \supp{U}| = 1$.
    \item $\supp{U} \subseteq X$: similar to the previous case, this would imply the contradiction
      $|\supp{C_x} \cap \supp{U}| = 1$ for all $x \in \supp{U}$.
    \item $\supp{U} = f \cup Y$ where $\emptyset \ne Y \subseteq X$: This case implies
      $f \in \supp{C_x}$ for all $x \in Y$ since otherwise there would exist an $x \in Y$
      such that $|\supp{C_x} \cap \supp{U}| = 1$. By using reorientation, we may assume that
      $C$ is positive. Note that this implies $C(f) = C_x(f) = C(x) = 1$ and $C_x(x) = -1$ for
      all $x \in Y$. By Lemma~\ref{lem:OrthImplUniSigInh} it follows that there exists an induced
      signing $U$ of the cocircuit $\supp{U}$. By replacing $U$ by $-U$
      if necessary, we may assume that $U(f) = 1$. Since $C \perp U$, there must exist an $x \in Y$
      such that $U(x) = -1$. But this contradicts $C_x \perp U$, since
      $\supp{C_x} \cap \supp{U} = \{f,x\}$.
  \end{enumerate*}
  It is thus possible to extend $X \cup f$ to a cobasis of $M|G$ to obtain a basis $\BC$
  of $M|G$ such that $\BC \subseteq G \setminus (X \cup f)$. Finally, the fundamental
  circuit $\supp{D}$ of $f$ with respect to $\BC$ satisfies $f \in \supp{D}, \supp{D}
  \subseteq \left( \supp{C} \cup \bigcup_{x \in X} \supp{C_x} \right) \setminus X$, and
  by Lemma~\ref{lem:OrthImplUniSigInh} either $D(f) = C(f)$ or $(-D)(f) = C(f)$ for an induced
  signing $D$ of $\supp{D}$.
\end{proof}

\begin{rem}
  In the context of ordinary matroids and with respect to the circuit axiom (C3),
  Proposition~\ref{prop:OrthImplSpecialCircElim} states the following: the range of elements $f$
  that can be picked and kept during (ordinary) strong circuit elimination is potentially enlarged
  by orthogonality since
  $\supp{C} \setminus \left( \bigcup_{x \in X} \supp{C_x} \right)
  \subseteq \supp{C} \setminus \left( \bigcup_{x \in X} \sep(C, C_x) \right)$. However, the price
  for this enlargement is the potential reduction of the set of elements that can be eliminated
  at once as
  $X \subseteq \bigcup_{x \in X} \sep(C, C_x) \subseteq \bigcup_{x \in X} \supp{C} \cap \supp{C_x}$.
\end{rem}

\begin{rem}
  Of course, if we restrict the setting of Proposition~\ref{prop:OrthImplSpecialCircElim} to finite
  orthogonally oriented matroids, it is known that among all such circuits $D$ there must be at
  least one that additionally satisfies
  $D^+ \subseteq \left( C^+ \cup \bigcup_{x \in X} C_x^+ \right) \setminus X$ and
  $D^- \subseteq \left( C^- \cup \bigcup_ {x \in X} C_x^- \right) \setminus X$ (see
  \citep[Chapter 3]{orimatsbook}). Unfortunately, the same statement for infinite matroids cannot
  be obtained by following the steps of the proof of Proposition~\ref{prop:OrthImplSpecialCircElim}:
  the circuit $D$ one obtains does in general not have this property, as the following example shows.
\end{rem}

\begin{exmp}\label{exmp:NonConfCircAfterElim}
  Consider the digraph shown in Figure~\ref{fig:ExaNonConfCircAfterElim} and its associated (finite)
  orthogonally oriented matroid with underlying matroid $M$ on the ground set
  $E = \{e_1,e_2,e_3,e_4,e_5,e_6\}$.
  \begin{figure}[h]
    \centering
    \begin{tikzpicture}[auto,node distance=1.8cm,
      dot/.style={draw,shape=circle,fill=black,scale=0.5}]
      \node[dot] (v1) {};
      \node[dot] (v2) [right=of v1] {}
      edge [-,postaction={decoration={markings,mark=at position 0.5 with {\arrow[scale=1.2]{<}}},decorate}]
      node {$e_1$} (v1);
      \node[dot] (v3) [right=of v2] {}
      edge [-,postaction={decoration={markings,mark=at position 0.5 with {\arrow[scale=1.2]{<}}},decorate}]
      node {$e_2$} (v2)
      edge [-,bend left=45,
      postaction={decoration={markings,mark=at position 0.5 with {\arrow[scale=1.2]{<}}},decorate}]
      node {$e_6$} (v1);
      \node[dot] (v4) [right=of v3] {}
      edge [-,postaction={decoration={markings,mark=at position 0.5 with {\arrow[scale=1.2]{<}}},decorate}]
      node {$e_3$} (v3)
      edge [-,bend right=45,
      postaction={decoration={markings,mark=at position 0.5 with {\arrow[scale=1.2]{>}}},decorate}]
      node {$e_5$} (v2)
      edge [-,bend left=90,thick,
      postaction={decoration={markings,mark=at position 0.5 with {\arrow[scale=1.2]{>}}},decorate}]
      node {$e_4$} (v1);
    \end{tikzpicture}
    \caption{Digraph of Example~\ref{exmp:NonConfCircAfterElim}\label{fig:ExaNonConfCircAfterElim}}
  \end{figure}
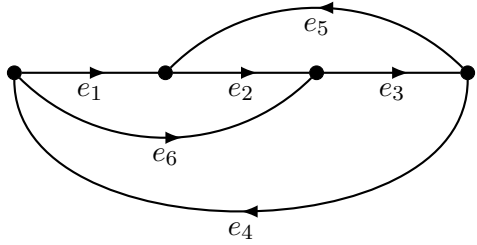
  It contains the signed circuits $C = (1,1,1,1,0,0)$ and $C_{e_1} = (-1,0,1,0,1,1)$. We choose
  $f = e_2, X = \{e_1\}$, and follow the steps of the proof of
  Proposition~\ref{prop:OrthImplSpecialCircElim}: First, we extend $X \cup f = \{e_1,e_2\}$ to a
  cobasis of $M$, say $\{e_1,e_2,e_3\}$. Then we have to add $e_2$ to the basis $B = \{e_4,e_5,e_6\}$
  to obtain the $B$-fundamental circuit $\supp{D} = \{e_2,e_4,e_5,e_6\}$. The induced signings of
  this circuit are $D = (0,1,0,-1,1,-1)$ and $-D = (0,-1,0,1,-1,1)$, i.e.\ both $D^-$ and $-D^-$ are
  non-empty. But since $\left( C^- \cup C_{e_1}^- \right) \setminus X = \emptyset$, neither of the
  inclusions $D^- \subseteq \left( C^- \cup \bigcup_ {x \in X} C_x^- \right) \setminus X$ and
  $-D^- \subseteq \left( C^- \cup \bigcup_ {x \in X} C_x^- \right) \setminus X$ is satisfied.
\end{exmp}

Let us return to the setting of Proposition~\ref{prop:OrthImplSpecialCircElim} for one last time:
assume that $D$ is a circuit for which at least one of the inclusions
$D^+ \subseteq \left( C^+ \cup \bigcup_{x \in X} C_x^+ \right) \setminus X$ and
$D^- \subseteq \left( C^- \cup \bigcup_ {x \in X} C_x^- \right) \setminus X$ fails, and let $e$
be one of the elements of $D$ that makes the respective inclusion fail. Under certain conditions,
it is then possible to find another circuit in $\supp{C} \cup \bigcup_{x \in X} \supp{C_x}$ that
neither includes $X$ nor the offending element $e$.

\begin{prop}
  Let $M, C, X, (C_x\ |\ x \in X), f$, and $D$ be as in Proposition~\ref{prop:OrthImplSpecialCircElim}.
  Denote the set $\left( C^+ \cup \bigcup_{x \in X} C_x^+ \right)$ by $A$ and the set
  $\left( C^- \cup \bigcup_{x \in X} C_x^- \right)$ by $B$. If there exists an
  $e \in (D^- \cap (A \setminus B)) \cup (D^+ \cap (B \setminus A))$,
  then there exists
  a $D' \in \CC$ such that $f \in \supp{D'}, D'(f) = C(f)$, and $\supp{D'}
  \subseteq \left( \supp{C} \cup \bigcup_{x \in X} \supp{C_x} \right)
  \setminus \left( X \cup e \right)$.
\end{prop}
\begin{proof}
  As in the proof of Proposition~\ref{prop:OrthImplSpecialCircElim}, we let
  $G = \left( \supp{C} \cup \bigcup_{x \in X} \supp{C_x} \right)$, and show that there exists
  a circuit $\supp{D'} \in \CC(M|G)$ that has the desired properties. From the proof of
  Proposition~\ref{prop:OrthImplSpecialCircElim}, we know that $X \cup f$ is coindependent in
  $M|G$. Thus it suffices to prove that $X \cup f$ stays coindependent in $M|G$ if we add $e$
  to this set. Assume for a contradiction that $X \cup \{e,f\}$ contains a cocircuit $\supp{U}$
  of $M|G$. Then one of the following two cases must hold.
  \begin{enumerate*}
    \item $\supp{U} \subseteq e \cup Y$ where $Y \subseteq X$: Since
      $|\supp{D} \cap \supp{U}| = 1$ if $e \in \supp{U}$, we find that
      $\supp{U} \subseteq Y \ne \emptyset$. But then $|\supp{C_x} \cap \supp{U}| = 1$ for every
      $x \in Y$, a contradiction again.
    \item $\supp{U} = Y \cup \{e,f\}$ where $\emptyset \ne Y \subseteq X$: Let $U$ be an induced
      signing of the cocircuit $\supp{U}$. From $U \perp D$ and
      $\supp{D} \cap \supp{U} = \{e,f\}$, it follows that $D(f)U(f) = -D(e)U(e)$. Since
      $C \perp U$, there must exist an $x \in Y \cup e$ such that $C(f)U(f) = -C(x)U(x)$. If $x = e$,
      then $C(e) = -D(e)$ and $C(f)U(f) = -C(e)U(e) = D(e)U(e) = -D(f)U(f) = -C(f)U(f)$, a
      contradiction. Thus $x \in Y$ must hold. Since $x \in \supp{C_x} \cap \supp{U}
      \subseteq \{e,f,x\}$, the circuit $C_x$ includes $e, f$ or both $e$ and $f$. Because of
      $C_x \perp U$ we conclude that $C_x(x)U(x) = -C_x(e)U(e)$ or $C_x(x)U(x) = -C_x(f)U(f)$. In
      case of the former $C_x(e) = -D(e)$ by the choice of $e$ and thus
      $C_x(x)U(x) = D(e)U(e) = -D(f)U(f) = -C(f)U(f) = C(x)U(x) = -C_x(x)U(x)$, a contradiction. In
      case of the latter $C_x(x)U(x) = -C_x(f)U(f) = -C(f)U(f) = C(x)U(x) = -C_x(x)U(x)$, a
      contradiction again.
  \end{enumerate*}
  It follows that $X \cup \{e,f\}$ is coindependent in $M|G$. Now we can proceed analogously
  to the proof of Proposition~\ref{prop:OrthImplSpecialCircElim} in order to show that there exists
  a signed circuit $D'$ of $M$ that has the desired properties.
\end{proof}

\section{An Axiom System Based on the Farkas Property}\label{sec:FPAxsioms}
The differing results of Sections~\ref{sec:OrthAxioms} and \ref{sec:NoPlainCircAxioms}
suggest the following: when looking for possible candidates of properties of circuit signatures
that might lead to the definition of a sensible class of infinite oriented matroids, it might
be more fruitful to consider those that explicitly or implicitly make statements about pairs
of circuit and cocircuit signatures. Addressing the problem of duality becomes easier in this
case. The so called {\em Farkas property} of pairs of sets of signed subsets
is an example of such a property.

\begin{defn}\label{defn:FarkasProp}
  A pair $\SC, \TC$ of sets of signed subsets of a set $E$ has the \textbf{Farkas property} if it
  satisfies the following condition:
  \begin{enumerate*}
    \item[(FP)] For all $e \in E$, either there is an $X \in \SC$ such that $e \in \supp{X}$ and $X$
    is positive, or there is a $Y \in \TC$ such that $e \in \supp{Y}$ and $Y$ is positive, but not
    both.
  \end{enumerate*}
\end{defn}

\subsection{FP-Orientable Matroids}\label{ssec:FPOriMats}
In order to use the Farkas property (FP) in the definition of oriented matroids, the following two steps
are necessary. First, it must be decided to which pairs of sets of signed subsets it relates to. Second,
the property has to be combined with a statement about minors of matroids. Bachem and Kern demonstrate
this approach in relation to finite oriented matroids in \citep[Section 5.3]{bachemlpd}, for example.
Since the definition of oriented matroids employed by them is not so prevalent in the literature, we
 briefly state it here in our notation (cf. \citep[Definitions 5.6 and 5.8]{bachemlpd}). In particular,
 it is worth noting that it explicitly does not mention an underlying ordinary matroid.

\begin{defn}\label{defn:BKOriMats}
  Let $\SC, \TC$ be a pair of sets of signed subsets of a finite set $E \ne \emptyset$, and let
  $F, G \subseteq E$ such that $F \cap G = \emptyset$. Then
  \[
    \SC/F\!\setminus\!G = \left\{ X|_{E\setminus F \cup G}\ \left|\ X \in \SC,\
                                                                     \supp{X} \subseteq E\setminus G \right.\right\}
  \]
  is called a \textbf{minor of $\bm{\SC}$} and $(\SC/F\!\setminus\!G, \TC/G\!\setminus\!F)$ is called a
  \textbf{minor of $\bm{(\SC, \TC)}$}. The pair $(\SC, \TC)$ is called an \textbf{oriented matroid on
  $\bm{E}$} if and only if $(\SC, \TC)$ satisfies the following conditions:
  \begin{enumerate*}
    \item $\SC$ and $\TC$ are symmetric.
    \item Every reorientation of every minor of $(\SC, \TC)$ has the Farkas property (FP).
  \end{enumerate*}
\end{defn}

\begin{rem}
  To connect Definition~\ref{defn:BKOriMats} to the notion of oriented matroids introduced earlier, we note
  that if $(\SC, \TC)$ is an oriented matroid in the sense of Bachem and Kern then the elements of $\SC$
  with minimal non-empty supports form the circuits of an oriented matroid, of which the elements of $\TC$
  with minimal non-empty support form the cocircuits.
\end{rem}

Now, on the basis of Definition~\ref{defn:BKOriMats}, in this section, we show how the Farkas Property can
be used to define another class of infinite oriented matroids. We start our endeavor with the following
rather technical definition which connects top level circuit signatures with signed circuits of minors.

\begin{defn}\label{defn:SSIndByCircSig}
  Let $M$ be a matroid on $E$, let $F, G \subseteq E$ such that $F \cap G = \emptyset$, and let
  $N = M/F\!\setminus\!G$ be a minor of $M$. For a consistent notation we denote the set of ordinary
  circuits of $N$ by $\supp{\CC(N)}$ and the set of ordinary cocircuits by $\supp{\CC^*(N)}$.
  Let $\CC, \CC^*$ be a pair of circuit and cocircuit signatures of $M$. Then $\CC$ and $\CC^*$ induce
  the following sets of signed subsets of $E(N)$:
  \[
    \SC^\CC(N) = \left\{ C|_{E(N)}\ \left|\ C \in \CC,\ \supp{C|_{E(N)}} \in \supp{\CC(N)},\
                                            \supp{C} \subseteq E(N) \cup F \right.\right\},
  \]
  \[
    \SC^{\CC^*}(N) = \left\{ U|_{E(N)}\ \left|\ U \in \CC^*,\ \supp{U|_{E(N)}} \in \supp{\CC^*(N)},\
                                                \supp{U} \subseteq E(N) \cup G \right.\right\}
  \]
\end{defn}

\begin{rem}\label{rem:SSIndByCircSig}
  Let $\SC^\CC(N), \SC^{\CC^*}(N)$ be two sets of signed subsets according to Definition~\ref{defn:SSIndByCircSig}.
  \begin{enumerate*}
    \item By Lemma~\ref{lem:StructCircOfMinors} the supports of the elements of $\SC^\CC(N)$ are
      exactly the circuits of $N$, i.e.\ $\{ \supp{X}\ |\ X \in \SC^\CC(N) \} = \supp{\CC(N)}$.
    \item In general, $\SC^\CC(N)$ is not a circuit signature of $N$. If $N$ is a proper minor of $M$,
      then it is possible that $\SC^\CC(N)$ contains two signed subsets $X,Y$ of $E(N)$ such that
      $\supp{X} = \supp{Y}$ but $X \ne Y, -Y$. On the other hand, if $\CC$ belongs to a pair of
      circuit and cocircuit signatures providing an orthogonal orientation of $M$ on $E$ according
      to Definition~\ref{defn:OrthOrient}, then Lemma~\ref{lem:OrthImplUniSigInh} implies that
      $\SC^\CC(N)$ is indeed a circuit signature of $N$.
    \item Dually, the statements (1) and (2) hold for $\SC^{\CC^*}(N)$ with respect to cocircuits,
      coscrawls, and cocircuit signatures of $N$.
  \end{enumerate*}
\end{rem}

Pairs built by the sets of signed subsets as we have just defined them in Definition~\ref{defn:SSIndByCircSig}
are the pairs we want to have the Farkas property (FP). This leads us to the following definition which
yields a class of orientable matroids different from that discussed in Section~\ref{sec:OrthAxioms}, as
we show in the course of this section.

\begin{defn}\label{defn:FPAxiom}
  Let $M$ be a matroid on a set $E$. Let $\CC$ be a circuit signature of $M$ and $\CC^*$ be a
  cocircuit signature of $M$. Then the triple $\MC = (M, \CC, \CC^*)$ is called a \textbf{Farkas
  property oriented matroid (on $\bm{E}$)}, or \textbf{FP-oriented matroid (on $\bm{E}$)}
  for short, if and only if $\CC$ and $\CC^*$ satisfy the following condition:
  \begin{enumerate*}
    \item[(FA)] If $N$ is a minor of $M$ with ground set $E(N)$ and $A \subseteq E(N)$, then
      the pair $_{-A}\SC^\CC(N), _{-A}\SC^{\CC^*}(N)$ has the Farkas property (FP).
  \end{enumerate*}
  In this case we say that $M$ is \textbf{FP-orientable (on $\bm{E}$)} and that the pair
  $\CC, \CC^*$ \textbf{provides an FP-orientation of $\bm{M}$ (on $\bm{E}$)}. The matroid $M$ is again
  referred to as the \textbf{underlying matroid of $\bm{\MC}$}. Finally, we define the \textbf{vectors
  of $\bm{\MC}$} as the vectors of $M$ with respect to $\CC$ and accordingly the \textbf{covectors of
  $\bm{\MC}$} as the covectors of $M$ with respect to $\CC^*$.
\end{defn}

\begin{rem}\label{rem:AboutStrongCondDefFPOriMats}
  Let $M$ be a matroid on $E$ and $\CC, \CC^*$ be a pair of circuit and cocircuit signatures of $M$. Let $N$
  be a minor of $M$. Note that in contrast to finite oriented matroids and the more general approach chosen
  in \citep[Section 5.3]{bachemlpd}, the definition of FP-oriented matroids explicitly refers to signed
  circuits/cocircuits of minors via the sets $\SC^\CC(N)$ and $\SC^{\CC^*}(N)$. The reason for this is that
  we want FP-oriented matroids to have the following property: the pair $\CC, \CC^*$ as well as every pair
  of circuit and cocircuit signatures induced by $\CC$ and $\CC^*$ on minors of $M$ should have the
  Farkas property (FP). As we will see in the course of Sections~\ref{ssec:PPForSSS} and
  \ref{ssec:MinorsAndPropsFPOriMats}, the strong conditions contained in Definition~\ref{defn:SSIndByCircSig}
  guarantee that this is the case.\\
  Nevertheless, we consider it an open question whether Definition~\ref{defn:SSIndByCircSig} can be weakened
  in such a way that the aforementioned property is still fulfilled. We briefly discuss two possible weakenings
  to illustrate the problems that may arise. On the one hand, one could replace the sets $\SC^\CC(N)$ and
  $\SC^{\CC^*}(N)$ by the following two sets where $\VC$ and $\WC$ denote the set of vectors and covectors of
  $M$ according to $\CC$ and $\CC^*$, respectively:
    \[
      \SC^\VC(N) = \left\{ X|_{E(N)}\ \left|\ X \in \VC,\
                                              \supp{X} \subseteq E(N) \cup F \right.\right\},
    \]
    \[
      \SC^\WC(N) = \left\{ Y|_{E(N)}\ \left|\ Y \in \WC,\
                                              \supp{Y} \subseteq E(N) \cup G \right.\right\}.
    \]
  As we will see in Example~\ref{exmp:CMWithoutFarkasCont} below, this weakening of Definition~\ref{defn:SSIndByCircSig}
  can lead to the situation that the specific pair $\SC^\VC(M), \SC^\WC(M)$ has the Farkas property (FP)
  while the pair $\CC, \CC^*$ does not. This means that the Farkas property (FP) for vectors and covectors does not
  imply (FP) for circuit and cocircuit signatures.\\
  On the other hand, one could drop the conditions $\supp{C|_{E(N)}} \in \supp{\CC(N)}$ and
  $\supp{U|_{E(N)}} \in \supp{\CC^*(N)}$ in the original definition of the sets $\SC^\CC(N)$ and $\SC^{\CC^*}(N)$,
  respectively, and consider the following two sets instead:
  \[
    \tilde{\SC}^\CC(N) = \left\{ C|_{E(N)}\ \left|\ C \in \CC,\
                                                    \supp{C} \subseteq E(N) \cup F \right.\right\},
  \]
  \[
    \tilde{\SC}^{\CC^*}(N) = \left\{ U|_{E(N)}\ \left|\ U \in \CC^*,\
                                                        \supp{U} \subseteq E(N) \cup G \right.\right\}.
  \]
  We suspect that this weakening of Definition~\ref{defn:SSIndByCircSig} could instead lead to the situation where
  the pair $\CC, \CC^*$ has the Farkas property (FP) while the pair of circuit and cocircuit signatures induced
  by $\CC$ and $\CC^*$ on a specific minor $N$ does not. However, we also consider it an open question whether an
  example showcasing this behavior exists. We make this question precise as Open Question~\ref{qstn:4PImplFA} below.
\end{rem}

\subsection{A Painting Property of Sets of Signed Subsets}\label{ssec:PPForSSS}
Before we continue to examine the properties of FP-oriented matroids, we supplement their
definition with a specific instance of a so called {\em painting property} of pairs of sets
of signed subsets. Properties labeled like this are in general more technically oriented and can
be found in graduated versions in the literature (see for instance \citep[Section 3.4]{orimatsbook}).
This supplementation has the following benefits: First, it makes the definition of FP-oriented
matroids as well as the discussion of examples of such matroids more accessible. Second, it
considerably simplifies the proofs of statements about FP-oriented matroids, including the claim
that the class of FP-oriented matroids is closed under taking minors.

\begin{defn}\label{defn:SS4PaintProp}
  A pair $\SC, \TC$ of sets of signed subsets of a set $E$ has the \textbf{4-painting property}
  if and only if it satisfies the following condition:
  \begin{enumerate*}
    \item[(4P)] For all 4-partitions $E = B \dot{\cup} W \dot{\cup} G \dot{\cup} R$ and
    $e \in B \cup W$, either there is an $X \in \SC$ such that
    $e \in \supp{X} \subseteq B \cup W \cup G, \supp{X} \cap B \subseteq X^+$, and
    $\supp{X} \cap W \subseteq X^-$, or there is a $Y \in \TC^*$ such that
    $e \in \supp{Y} \subseteq B \cup W \cup R,
    \supp{Y} \cap B \subseteq Y^+$, and $\supp{Y} \cap W \subseteq Y^-$, but not
    both.
  \end{enumerate*}
\end{defn}

Next, we establish a connection between FP-oriented matroids and the 4-painting property (4P).
Similar to the finite case, the pair of circuit and cocircuit signatures of an FP-oriented
matroid has the 4-painting property (4P).

\begin{lem}\label{lem:FPOrientationsImpl4P}
  Let $(M, \CC, \CC^*)$ be a Farkas property oriented matroid on $E$. Then the pair $\CC, \CC^*$
  has the 4-painting property (4P).
\end{lem}
\begin{proof}
  Let $E = B \dot{\cup} W \dot{\cup} G \dot{\cup} R$ be a 4-partition of $E$ and
  $e \in B \dot{\cup} W$. First, we show that at least one of the alternatives of (4P) holds.
  Let $N$ be the minor $M/G\!\setminus\!R$ of $M$. Then the pair of signed subsets
  $\SC = _{-W}\SC^\CC(N), \TC = _{-W}\SC^{\CC^*}(N)$ has the Farkas property (FP). Thus, either
  there is a positive $C' \in \SC$ such that $e \in \supp{C'} \subseteq B \cup W$, or there is a
  positive $U' \in \TC$ such that $e \in \supp{U'} \subseteq B \cup W$, but not both. By definition
  of $\SC$ and $\TC$, this means the following. In the first case, there exists a $C \in \CC$ such
  that $\supp{C} \subseteq \supp{C'} \cup G \subseteq B \cup W \cup G,
  e \in \supp{C} \cap (B \cup W), \supp{C} \cap B \subseteq C^+$, and $\supp{C} \cap W \subseteq C^-$.
  In the second case, there exists a $U \in \CC^*$ such that
  $\supp{U} \subseteq \supp{U'} \cup R \subseteq B \cup W \cup R, e \in \supp{U} \cap (B \cup W),
  \supp{U} \cap B \subseteq U^+$, and $\supp{U} \cap W \subseteq U^-$. Thus it remains to show that
  at most one of the alternatives of (4P) holds. Assume for a contradiction that $C \in \CC$ and
  $U \in \CC^*$ satisfy both alternatives. Let $G' = \{ f \in G\ |\ C(f) = -1 \},
  R' = \{ f \in R\ |\ U(f) = -1 \}, \SC' = _{-(W \cup G' \cup R')}\SC^\CC(M)$, and
  $\TC' = _{-(W \cup G' \cup R')}\SC^{\CC^*}(M)$. Then $_{-(W \cup G')}C$ and $_{-(W \cup R')}U$
  both are positive elements of $\SC'$ and $\TC'$ containing $e$. This contradicts the fact that the
  pair $\SC', \TC'$ has the Farkas property (FP).
\end{proof}

In the finite case, it can be shown that the reverse statement of Lemma~\ref{lem:FPOrientationsImpl4P}
holds for oriented matroids. In fact, the previous lemma and a weaker version of the following lemma
are usually combined into a single proposition which is known as {\em Minty's Lemma} in the literature
(see for instance \citep[Propositon 5.12]{bachemlpd}). However, since the definition of FP-oriented
matroids $(M, \CC, \CC^*)$ explicitly refers to circuits and cocircuits of $M$, the pair $\CC, \CC^*$
must be augmented with a specific additional property for the reverse statement to hold.

\begin{lem}\label{lem:From4PToFPOrientations}
  Let $M$ be a matroid on a set $E$. Let $\CC$ be a circuit signature of $M$ and $\CC^*$
  be a cocircuit signature of $M$. Then $(M, \CC, \CC^*)$ is a Farkas property oriented matroid
  on $E$ if the pair $\CC, \CC^*$ satisfies the following conditions:
  \begin{enumerate}
    \item The pair $\CC, \CC^*$ has the 4-painting property (4P).
    \item Let $E = B \dot{\cup} W \dot{\cup} G \dot{\cup} R$ be a 4-partition of $E$ and
    $e \in B \dot{\cup} W$. If the first alternative of (4P) holds, then among all the
    $C \in \CC$ satisfying this alternative there is one such that $\supp{C|_{B \cup W}}$
    is a circuit of $M/G\!\setminus\!R$. Similarly, if the second alternative of (4P) holds,
    then among all the $U \in \CC^*$ satisfying this alternative there is one such that
    $\supp{U|_{B \cup W}}$ is a cocircuit of $M/G\!\setminus\!R$.
  \end{enumerate}
\end{lem}
\begin{proof}
  Let $N = M/G\!\setminus\!R$ be a minor of $M$, let $e \in E(N)$, and let $A \subseteq E(N)$. Then
  $(E(N)\setminus A) \dot{\cup} A \dot{\cup} G \dot{\cup} R$ is a 4-partition of $E$ and thus exactly
  one of the alternatives of (4P) holds for the pair $\CC, \CC^*$. If the first one holds, then
  according to property (2) there exists a $C \in \CC$ such that $_{-A}C|_{E(N)}$ is a positive circuit
  of $N$ and $e \in \supp{C|_{E(N)}}$. Otherwise there exists a $U \in \CC^*$ such that $_{-A}U|_{E(N)}$
  is a positive cocircuit of $N$ and $e \in \supp{U|_{E(N)}}$. Thus the pair
  $_{-A}\SC^\CC(N), _{-A}\SC^{\CC^*}(N)$ has the Farkas property (FP).
\end{proof}

\begin{rem}
  How strong the second condition of Lemma~\ref{lem:From4PToFPOrientations} must be is closely linked
  to the definition of FP-oriented matroids. See Remark~\ref{rem:AboutStrongCondDefFPOriMats} for a
  discussion about the possibility of weakening the definition of FP-oriented matroids and thus also
  the second condition of Lemma~\ref{lem:From4PToFPOrientations}.
\end{rem}

In fact, we do not know the answer to the following.

\begin{qstn}\label{qstn:4PImplFA}
  Is the second condition of Lemma~\ref{lem:From4PToFPOrientations} necessary or does the 4-painting
  property (4P) alone already imply the property (FA) for the pair $\CC, \CC^*$?
\end{qstn}

Our next goal in this section is to show that the 4-painting property (4P) implies orthogonality
and allows strong circuit elimination as well as conformal vector decomposition; we consider
orthogonality first.

\begin{lem}\label{lem:4PImplOrth}
  Let $\SC, \TC$ be a pair of sets of signed subsets of a set $E$. If $\SC, \TC$ has the painting
  property (4P) and $\TC$ is symmetric, then $X \perp Y$ for all $X \in \SC$ and $Y \in \TC$.
\end{lem}
\begin{proof}
  Let $X \in \SC, Y \in \TC$, and assume for a contradiction that $X$ and $Y$ are not orthogonal.
  By replacing $Y$ by $-Y$ if necessary, we may assume that $X(e) = Y(e)$ for all
  $e \in \supp{X} \cap \supp{Y}$. Choose
  $B = \{ f\ |\ f \in X^+ \cap Y^+ \}, W = \{ f\ |\ f \in X^- \cap Y^- \},
  G = \supp{X} \setminus (\supp{X} \cap \supp{Y}), R = E \setminus (B \cup W \cup G)$, and
  $e \in B \cup W \ne \emptyset$. Then both alternatives of (4P) hold, a contradiction. Thus
  $X \perp Y$ for all $X \in \SC$ and $Y \in \TC$
\end{proof}

\begin{cor}\label{cor:4PImplOrthOrient}
  Let $M$ be a matroid and $\CC, \CC^*$ be a pair of circuit and cocircuit signatures of $M$ that
  has the 4-painting property (4P). Then $(M, \CC, \CC^*)$ is an orthogonally oriented matroid.
\end{cor}
\begin{proof}
  Lemma~\ref{lem:4PImplOrth}.
\end{proof}

Pairs of circuit and cocircuit signatures of a matroid $M$ having the 4-painting property (4P) thus
induce circuit and cocircuit signatures on the minors of $M$, a fact that is a part of the proof of
the claim that the class of FP-oriented matroids is closed under taking minors.\\

Next, we show that the 4-painting property (4P) allows strong circuit elimination. Of course, in the
dual sense this means that strong cocircuit elimination is possible as well. Since we want to make
use of (4P), the key point in the proof of the following lemma is to consider suitable 4-partitions
of the ground set of a given matroid.

\begin{lem}\label{lem:4PImplSSCE}
  Let $M$ be a matroid and $\CC, \CC^*$ be a pair of circuit and cocircuit signatures of $M$ that
  has the 4-painting property (4P). Then both $\CC$ and $\CC^*$ have the strong circuit elimination
  property (CE).
\end{lem}
\begin{proof}
  Let $C \in \CC, X \subseteq \supp{C}$, and $(C_x\ |\ x \in X)$ be
  a family of elements of $\CC$ such that $\supp{C_x} \cap X = \{x\}$ and
  $x \in \sep(C, C_x)$ for all $x \in X$, and let
  $f \in \supp{C} \setminus \left( \bigcup_{x \in X} \sep(C, C_x) \right)$. Choose
  \[
    G = \left( \bigcup\limits_{x \in X} \sep(C,C_x)
               \cup \bigcup\limits_{x,y \in X} \sep(C_x,C_y) \right) \setminus X,
  \]
  \[
    B = \left( C^+ \cup \bigcup\limits_{x \in X} C_x^+ \right) \setminus (G \dot{\cup} X),
  \]
  \[
    W = \left( C^- \cup \bigcup\limits_{x \in X} C_x^- \right) \setminus (G \dot{\cup} X),
  \]
  \[
    R = \left( E \setminus \left( \supp{C} \cup \bigcup_{x \in X} \supp{C_x} \right) \right)
        \dot{\cup} X.
  \]
  Then $E = B \dot{\cup} W \dot{\cup} G \dot{\cup} R$ and $f \in B \cup W$.
  Assume for a contradiction that the second alternative of (4P) holds, i.e.\ that there is a
  $U \in \CC^*$ such that $f \in \supp{U} \subseteq B \cup W \cup R,
  \supp{U} \cap B \subseteq U^+$, and $\supp{U} \cap W \subseteq U^-$. Then
  $\supp{C} \cap \supp{U} \subseteq B \cup W \cup X$ and
  $\supp{C_x} \cap \supp{U} \subseteq B \cup W \cup x$ for all $x \in X$. By Lemma~\ref{lem:4PImplOrth}
  we have $C \perp U$ and $C_x \perp U$ for all $x \in X$. Since additionally $C(e) = U(e)$
  holds for all $e \in (B \cup W) \cap \supp{C}$, there must exist an $x \in X$ such that $x \in \supp{U}$
  and $C(x) = -U(x) = -C_x(x)$. But this implies $C_x(e) = U(e)$ for all $e \in \supp{C_x} \cap \supp{U}$
  in contradiction to $C_x \perp U$. Thus the first alternative of (4P) must hold, i.e.\ there is a
  $D \in \CC$ such that $f \in \supp{D} \subseteq B \cup W \cup G, \supp{D} \cap B \subseteq D^+$, and
  $\supp{D} \cap W \subseteq D^-$. This $D$ then satisfies the conditions of the property (CE). Finally,
  by dualizing this argument it follows that $\CC^*$ has the property (CE) as well.
\end{proof}

Finally, we address the claim that decomposing vectors/covectors into conforming circuits/cocircuits is
possible when the 4-painting property (4P) is available.

\begin{prop}\label{prop:4PImplConfDecompOfOrthSS}
  Let $M$ be a matroid that is equipped with a pair of circuit and cocircuit signatures that
  has the 4-painting property (4P). Let $X \subseteq \TPME$ such that $X$ is orthogonal to every
  cocircuit of $M$. Then there exists a decomposition of $X$ into signed circuits of $M$ conforming
  to $X$.
\end{prop}
\begin{proof}
  By using reorientation we may assume that $X$ is positive. Let
  $E = B \dot{\cup} W \dot{\cup} G \dot{\cup} R$ be a 4-partition of $E$ where
  $B = \supp{X}, W = G = \emptyset$, and $R = E \setminus B$. Let $e \in \supp{X}$. Then, according
  to Proposition~\ref{prop:OrthToAllElemThenOrthToComp}, the second alternative of (4P)
  cannot hold since every signed cocircuit of $M$ is orthogonal to $X$. Thus, for every $e \in \supp{X}$
  there is a $C_e \in \CC$ such that $e \in \supp{C_e}$ and $C_e$ conforms to $X$.
\end{proof}

\begin{cor}\label{prop:4PImplConfDecomp}
  Let $M$ be a matroid that is equipped with a pair of circuit and cocircuit signatures that
  has the 4-painting property (4P) and let $V$ be a vector of $M$. Then there exists a decomposition
  of $V$ into signed circuits of $M$ conforming to $V$.
\end{cor}
\begin{proof}
  Follows from Corollaries~\ref{cor:OrthOfVecsAndCovecsOfOrthOriMat} and \ref{cor:4PImplOrthOrient}
  as well as Proposition~\ref{prop:4PImplConfDecompOfOrthSS}.
\end{proof}

To conclude this section, note that we have just shown the following result.

\begin{cor}
  Let $M$ be a matroid that is equipped with a pair of circuit and cocircuit signatures that
  has the 4-painting property (4P) and let $X \subseteq \TPME$. Then the following are equivalent:
  \begin{enumerate*}
    \item $X$ is orthogonal to every signed cocircuit of $M$.
    \item $X$ is a vector of $M$.
  \end{enumerate*}
\end{cor}
\begin{proof}
  $(1) \Rightarrow (2)$ follows from Proposition~\ref{prop:4PImplConfDecompOfOrthSS}, while
  $(2) \Rightarrow (1)$ follows from Corollaries~\ref{cor:OrthOfVecsAndCovecsOfOrthOriMat} and
  \ref{cor:4PImplOrthOrient}.
\end{proof}

\subsection{Minors and Properties of FP-Oriented Matroids}\label{ssec:MinorsAndPropsFPOriMats}
Just like in the case of orthogonally oriented matroids, we note that an immediate consequence of
Definition~\ref{defn:FPAxiom} is that the dual matroid of an FP-orientable matroid is an
FP-orientable matroid as well.\\

We now show that the class of FP-oriented matroids is closed under taking minors. To do so, we
start by showing that the class of FP-oriented matroids is a subclass of the class of orthogonally
oriented matroids. This follows immediately from the results of Section~\ref{ssec:PPForSSS}.

\begin{cor}\label{cor:FPMatsSubclassOrthMats}
  Let $(M, \CC, \CC^*)$ be a Farkas property oriented matroid. Then $(M, \CC, \CC^*)$ is an
  orthogonally oriented matroid as well, i.e.\ the class of FP-oriented matroids is a subclass
  of the class of orthogonally oriented matroids.
\end{cor}
\begin{proof}
  Lemma~\ref{lem:FPOrientationsImpl4P} and Corollary~\ref{cor:4PImplOrthOrient}.
\end{proof}

The proof of the claim that the class of FP-oriented matroids is closed under taking minors
relies on the fact that circuit and cocircuit signatures of FP-oriented matroids induce circuit
and cocircuit signatures on minors.

\begin{cor}\label{cor:FPSignIndMinSign}
  Let $(M, \CC, \CC^*)$ be a Farkas property oriented matroid and let $N$ be a minor of $M$. Then
  $\SC^\CC(N), \SC^{\CC^*}(N)$ is a pair of circuit and cocircuit signatures of $N$.
\end{cor}
\begin{proof}
  Corollary~\ref{cor:FPMatsSubclassOrthMats} and Remark~\ref{rem:SSIndByCircSig}.
\end{proof}

Now we are ready to show that the class of FP-oriented matroids is closed under taking minors.

\begin{thm}\label{thm:MinorsFPAxiomsOMsAreFPAxiomsOMs}
  Let $(M, \CC, \CC^*)$ be a Farkas property oriented matroid and let $N$ be a minor of $M$.
  Then $(N, \SC^\CC(N), \SC^{\CC^*}(N))$ is a Farkas property oriented matroid. The class
  of FP-oriented matroids is thus closed under taking minors.
\end{thm}
\begin{proof}
  Denote $\SC^\CC(N)$ by $\CC_N$ and $\SC^{\CC^*}(N)$ by $\CC^*_N$, respectively. According to
  Corollary~\ref{cor:FPSignIndMinSign}, the pair $\CC_N, \CC^*_N$ is a pair of circuit and cocircuit
  signatures of $N$. Now, let $N'$ be a minor of $N$ (and thus of $M$). By
  Corollary~\ref{cor:FPMatsSubclassOrthMats}, the identities $\SC^{\CC_N}(N') = \SC^\CC(N')$ and
  $\SC^{\CC^*_N}(N') = \SC^{\CC^*}(N')$ hold. Thus the pair $_{-A}\SC^\CC(N'), _{-A}\SC^{\CC^*}(N')$
  has the Farkas property (FP) for all $A \subseteq E$.
\end{proof}

Like in the case of orthogonally oriented matroids, given a Farkas property oriented matroid
$\MC = (M, \CC, \CC^*)$ on a set $E$ and a subset $X \subseteq E$, it is thus sensible to speak of
the \textbf{FP-oriented restriction minor $\bm{\MC|X}$}, the \textbf{FP-oriented deletion minor
$\bm{\MC\!\setminus\!\overline{X}}$}, and the \textbf{FP-oriented contraction minor $\bm{\MC.X}$
or $\bm{\MC/\overline{X}}$ of $\bm{\MC}$}. Again, for any minor $N$ of $M$, we denote the circuit
and cocircuit signatures induced by $\CC$ and $\CC^*$ on $N$ by $\CC_N$ and $\CC^*_N$, respectively.
Finally, note that the alternative characterization of these minors given in
Subsection~\ref{ssec:MinorsAndPropsOrthOriMats} also applies in this case.\\

Next, we show that two well known operations common to all finite oriented matroids can be performed
over FP-oriented matroids as well, namely strong circuit elimination and conformal vector/covector
decomposition. Due to our findings in Section~\ref{ssec:PPForSSS} that the 4-painting property (4P)
already enables both operations, these two claims also follow immediately.

\begin{cor}\label{cor:FPOrientationsImplSSCE}
  Let $(M, \CC, \CC^*)$ be a Farkas property oriented matroid. Then both $\CC$ and $\CC^*$ have the
  strong circuit elimination property (CE).
\end{cor}
\begin{proof}
  Lemmas~\ref{lem:FPOrientationsImpl4P} and \ref{lem:4PImplSSCE}.
\end{proof}

\begin{cor}\label{cor:FPOriMatsAllowConfDecomp}
  Let $\MC$ be an FP-oriented matroid and let $V$ vector of $\MC$. Then there exists a
  decomposition of $V$ into signed circuits of $\MC$ conforming to $V$.
\end{cor}
\begin{proof}
  Lemma~\ref{lem:FPOrientationsImpl4P} and Corollary~\ref{prop:4PImplConfDecomp}.
\end{proof}

\begin{rem}
  Corollary~\ref{cor:FPOriMatsAllowConfDecomp} can be considered a rough analog of the
  {\em Composition Theorem} \citep[Theorem 5.36]{bachemlpd}.
\end{rem}

Since minors of an FP-oriented matroid inherit the 4-painting property (4P), strong circuit
elimination and conformal vector/covector decomposition are inherited as well. In that sense,
FP-orientability can thus be considered a rather strong feature in comparison to orthogonal
orientability.

\subsection{Examples and Counterexamples of FP-Orientable Matroids}\label{ssec:ExmpCexmpFPOriMats}
In this section, we consider examples of matroids that are FP-orientable as well as examples of
matroids that are not. However, unlike in the finite case the latter does not mean that those
matroids are not orthogonally orientable; in fact, the class of FP-oriented matroids is a proper
subclass of the class of orthogonally oriented matroids, as we shall see. But before we venture
into the territory of truly infinite FP-oriented matroids, we deal with the case of finite
FP-oriented matroids (i.e.\ ordinary oriented matroids) first.

\begin{exmp}\label{exmp:FinOriMatsAreFPOriMats}
  According to Lemma~\ref{lem:From4PToFPOrientations}, every finite oriented matroid $(M, \CC,
  \CC^*)$ is an FP-oriented matroid: Since the pair $\CC, \CC^*$ has the 4-painting property (4P),
  it remains to show that property (2) holds for $\CC, \CC^*$. Let
  $E = B \dot{\cup} W \dot{\cup} G \dot{\cup} R$ be a 4-partition of $E$, let $e \in B \cup W$,
  and assume that the first alternative of (4P) is satisfied. Consider the minor
  $N = M/G\!\setminus\!R$ and the pair of circuit and cocircuit signatures $\CC_N, \CC_N^*$
  of $N$ induced by $\CC$ and $\CC^*$, respectively. Since $(N, \CC_N, \CC_N^*)$ is an ordinary
  oriented matroid, (4P) holds for the pair $\CC_N, \CC_N^*$. Let $E(N) = B \dot{\cup} W$ and assume
  for a contradiction that there is a $U' \in \CC_N^*$ which satisfies the second alternative of (4P)
  with respect to $e \in B \cup W$. Since $U' = U|_{B \cup W}$ for some $U \in \CC^*$ with
  $\supp{U} \subseteq B \cup W \cup R$, the second alternative of (4P) then holds for
  $E = B \dot{\cup} W \dot{\cup} G \dot{\cup} R$ and $e \in B \cup W$ as well, a contradiction. Thus
  the first alternative of (4P) must hold for $E(N) = B \dot{\cup} W$. Any $C' \in \CC_N$ satisfying
  this alternative then yields a $C \in \CC$ with the desired properties. Dually, property (2) holds
  as well if the second alternative of (4P) is satisfied.
\end{exmp}

\subsubsection{The Farkas Property for Finitary and Cofinitary Regular Matroids}
Recall from Example~\ref{exmp:TameRegMats} that a tame regular matroid is signable, i.e. for each
circuit $C$ and for each cocircuit $U$ of a tame regular matroid $M$ there exists a choice of
functions $f_C \colon C \rightarrow \{-1,1\}$ and $g_U \colon U \rightarrow \{-1,1\}$, respectively,
such that
\[
  \sum_{e \in C \cap U} f_C(e) g_U(e) = 0,
\]
where the sum is evaluated over $\ZZ$. If $N$ is a minor of $M$, then we will write $f^M_C, g^M_U$
and $f^N_C, g^N_U$ below to distinguish between the signings of $M$ and $N$, if necessary.\\

Since it suffices to show that cofinitary regular matroids are FP-oriented matroids to derive the
dual statement that finitary regular matroids are FP-oriented matroids, we concentrate on cofinitary
regular matroids for the rest of this section.\\

We say a circuit $C$ of a tame regular matroid is {\em directed} if $f_C$ is constant; dually, we
say a cocircuit $U$ is {\em directed} if $g_U$ is constant. To see that cofinitary regular matroids
are FP-oriented matroids as well, we first prove that every element of such a matroid is either
contained in a directed circuit or in a directed cocircuit. For this purpose, we will need the
following classical theorem for finite regular matroids (see for instance \citet{mintycolandorimats}).

\begin{thm}\label{thm:MicroMenger}
  Let $M$ be a finite regular matroid and $e$ an element of $M$. Then either there is a directed
  circuit containing $e$ or else there is a directed cocircuit containing $e$. \qed
\end{thm}

A {\em flow} in a cofinitary regular matroid $M$ with ground set $E$ is a function $f \colon E \to \ZZ$
such that for any cocircuit $U$ of $M$ we have $\sum_{e \in U} f(e) g_U(e) = 0$. Thus for example
for any circuit $C$ the function obtained by extending $f_C$ to the whole of $E$ by 0s outside $C$
is a flow. We say a flow is {\em non-negative} if it only takes non-negative integral values. A
{\em flowset} is a subset $X$ of $E$ such that the characteristic function $\mathds{1}_X$ of $X$ is a
flow. Thus any directed circuit is a flowset.\\

To show that every element of a cofinitary regular matroid is either contained in a directed circuit
or in a directed cocircuit, we will need three more lemmas. The first one can be obtained from
Theorem~\ref{thm:MicroMenger} by a straightforward compactness argument.

\begin{lem}\label{lem:FlowMenger}
  Let $M$ be a cofinitary regular matroid and $e$ an element of $M$. Then either there is a flowset
  containing $e$ or else there is a directed cocircuit containing $e$. \qed
\end{lem}

The second lemma establishes a certain kind of strong conformal decomposition that is possible in
the context of tame regular matroids. In fact, this is one of the main reasons why the claim that
every element of a cofinitary tame matroid is either contained in a directed circuit or in a directed
cocircuit turns out to be true.

\begin{lem}\label{lem:DisjScrawlDecomp}
  Let $M$ be a regular matroid and $g \colon E \rightarrow \{-1,1\}$ a function with finite support
  $\supp{g} \ne \emptyset$ such that for every circuit $C$ of $M$
  \[
    \sum_{e \in C} f_C(e) g(e) = 0.
  \]
  Then there exists a decomposition of $\supp{g}$ into disjoint cocircuits $U_1, \ldots, U_k$ of $M$
  such that $\mathop{g}|_{U_i} = \pm \mathop{g}_{U_i}$.
\end{lem}
\begin{proof}
  Without loss of generality we may assume that $g$ is non-negative. Note that $\supp{g}$ is a
  coscrawl by Lemma~\ref{lem:ScrawlChar} since it cannot meet any circuit only once. We prove the
  claim by induction on the size of $\supp{g}$. If $\supp{g}$ consists of a single element, then it
  is a coloop. In this case the claim follows trivially. Now assume $|\supp{g}| > 1$, let
  $e \in \supp{g}$, and consider the minor $N = M.\supp{g}$. Since $N$ is a finite regular matroid,
  by Theorem~\ref{thm:MicroMenger} the element $e$ is either contained in the support of a directed
  circuit $C'$ or in the support of a directed cocircuit $U'$ of $N$ but not both. If the first
  alternative holds, then there exists a circuit $C$ of $M$ such that $f^M_C|_{C'} = f^N_{C'}$ and
  $\sum_{e' \in C} f_C(e')g(e') \ne 0$, a contradiction. So there exists a directed cocircuit $U'$
  of $N$ containing $e$. Since $g|_{U'} = \pm g^N_{U'} = \pm g^M_{U'}$ the function
  $g' = g \mp g^N_{U'}$ is non-negative, satisfies $\sum_{e' \in C} f_C(e') g'(e') = 0$ for every
  circuit $C$ of $M$, and has a support that is strictly contained in $\supp{g}$. Hence we can apply
  the induction hypothesis on $g'$ and are done.
\end{proof}

The third and last lemma shows that the supports of non-negative flows on cofinitary regular matroids
consist of collections of flowsets; it is heavily dependent on the previous lemma.

\begin{lem}\label{lem:InSupport}
  Let $f$ be a non-negative flow on a cofinitary regular matroid $M$ and let $e$ be an element in the support
  $\supp{f}$ of $f$. Then there is a flowset $X$ with $e \in X \subseteq \supp{f}$.
\end{lem}
\begin{proof}
  Let $N = M|{\supp{f}}$. By applying Lemma \ref{lem:FlowMenger} to $N$ and $e$, we obtain a flowset $X$ with
  $e \in X \subseteq E(N) = \supp{f}$, since there is certainly no directed cocircuit containing $e$. We claim
  that $\mathds{1}_X$ is a flow in $E(M)$ as well: Let $U$ be a cocircuit of $M$ such that
  $U \cap X \ne \emptyset$. We want to apply Lemma~\ref{lem:DisjScrawlDecomp} to $g' = g^M_U|_{\supp{f}}$ in
  $N$ so we have to check that its requirements are met. Let $C'$ be a circuit of $N$. Since
  $f^N_{C'} = f^M_{C'}$ we find indeed that
  \[
    \sum_{e' \in C'} f^N_{C'}(e) g'(e) = \sum_{e' \in C' \cap U} f^M_{C'}(e) g^M_U(e) = 0.
  \]
  Then by Lemma~\ref{lem:DisjScrawlDecomp} there are disjoint cocircuits $U_1, \ldots, U_k$ of $N$ such that
  $\supp{g'} = U \cap \supp{f} = U_1 \cup \ldots \cup U_k$ and
  $g'|_{U_i} = \mathop{g}_U^M|_{U_i} = \pm \mathop{g}_{U_i}^N$
  for $1 \leq i \leq k$. Thus
  \[
    \sum_{e \in X \cap U} g_U^M(e)
    = \sum_{i = 1}^{k} \sum_{e \in X \cap U_i} g_{U_i}^N(e)
    = 0
  \]
  since
  \[
    \sum_{e \in X \cap U_i} g_{U_i}^N(e) = 0
  \]
  for $1 \leq i \leq k$; this concludes the proof.
\end{proof}

The two Lemmas~\ref{lem:FlowMenger} and \ref{lem:InSupport} finally alow us to derive the desired result
about cofinitary regular matroids.

\begin{thm}\label{thm:MiniMenger}
  Let $M$ be a cofinitary regular matroid and $e$ an element of $M$. Then either there is a directed circuit
  containing $e$ or else a directed cocircuit containing $e$.
\end{thm}
\begin{proof}
  Suppose that there is no directed cocircuit containing $e$. Then there is some flowset containing $e$
  by Lemma~\ref{lem:FlowMenger}. Let $X$ be a minimal such flowset (such an $X$ exists by Zorn's Lemma).
  Since it is a flowset, $X$ cannot meet any cocircuit just once. So by Lemma~\ref{lem:ScrawlChar} it is
  a scrawl. Let $C$ be a circuit with $e \in C \subseteq X$. Without loss of generality we may assume
  that $f_C(e) = 1$. Suppose for a contradiction that $C$ is not directed. Then there is some
  $e' \in C$ with $f_C(e') = -1$. Let $f$ be obtained by extending $f_C$ to the whole of $E$
  by 0s outside $C$. Then $f + \mathds{1}_X$ is a non-negative flow, so by Lemma~\ref{lem:InSupport}
  there is a flowset $X'$ with $e \in X' \subseteq \supp{f + \mathds{1}_X} \subseteq X - e'$,
  contradicting the minimality of $X$. Thus $C$ is a directed circuit containing $e$, completing the
  proof.
\end{proof}

\begin{cor}\label{cor:MiniMengerFinRegMats}
  Let $M$ be a finitary regular matroid and $e$ an element of $M$. Then either there is a directed circuit
  containing $e$ or else a directed cocircuit containing $e$.
\end{cor}
\begin{proof}
  Theorem~\ref{thm:MiniMenger}.
\end{proof}

\begin{exmp}\label{exmp:CofinRegMatsAreFPOriMats}
  Let $M$ be a cofinitary regular matroid. Then $M$ is FP-orientable: First, note that every
  minor of $M$ is a cofinitary regular matroid as well. Second, every reorientation of a signing
  is a signing again. Finally, every signing induces a pair $\CC, \CC^*$ of circuit and cocircuit
  signatures of $M$ (see for instance Example~\ref{exmp:TameRegMatsAreOrthOriMats}). Now, let $N$
  be a minor of $M$ and denote the circuit and cocircuit signatures induced on $N$ by the signing
  of $M$ by $\CC_N$ and $\CC^*_N$, respectively. Then $\CC_N = \SC^\CC(N)$ and
  $\CC^*_N = \SC^{\CC^*}(N)$. Since the pair $\SC^\CC(N), \SC^{\CC^*}(N)$ has the Farkas
  property (FP) by Theorem~\ref{thm:MiniMenger}, it follows that $(M, \CC, \CC^*)$ is an
  FP-oriented matroid.\\
  Dually, if $M$ is a finitary regular matroid, then by Corollary~\ref{cor:MiniMengerFinRegMats}
  and a similar chain of arguments it follows that $M$ is FP-orientable as well.
\end{exmp}

\subsubsection{The Farkas Property for Algebraic Cycle Matroids}
Let $G$ be a graph with vertex set $V$. The {\em elementary algebraic cycles} of $G$ are the edge
sets of the finite cycles in $G$ and those of its double rays. As Higgs shows in
\citep{infgraphsmatroids}, the elementary algebraic cycles of $G$ are the circuits of a matroid
on its edge set $E(G)$ if and only if $G$ includes no subdivision of the Bean Graph. This matroid
is called the {\em algebraic cycle matroid} of $G$ and is denoted by $M_{AC}(G)$. It is tame (see
for instance for the proof of Theorem~4.14 in \citep{tsmandduality}) and its dual is examined in
more detail in \citep[Section 5]{infmatsingraphs}. As for its cocircuits: If $X$ is a set of
vertices, then the corresponding {\em cut} $U_X$ is the set of edges with one end in $X$ and the
other in $V \setminus X$. If $G[X]$ is rayless, then the cut $U_X$ is called a {\em nibble}. The
minimal non-empty nibbles are the cocircuits of $M_{AC}(G)$. Any nibble is a disjoint union of such
cocircuits, and any directed nibble is a disjoint union of directed cocircuits.\\

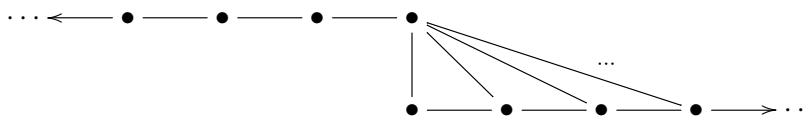
\begin{figure}[h]
\[
  \xymatrix{
    \cdots & \ar[l] \bullet \ar@{-}[r] & \bullet \ar@{-}[r] & \bullet \ar@{-}[r] & \bullet \ar@{-}[d]
                                                                                           \ar@{-}[dr]
                                                                                           \ar@{-}[drr]
                                                                                           \ar@{-}[drrr]
                                                                                           \ar@{}[drrrr]|{\cdots} &                    &                    &                &\\
           &                           &                    &                    & \bullet \ar@{-}[r]             & \bullet \ar@{-}[r] & \bullet \ar@{-}[r] & \bullet \ar[r] & \cdots
  }
\]
\caption{The Bean Graph}
\end{figure}

For the rest of this section we will always assume that $G$ is the underlying graph of some directed
graph that contains no subdivision of the Bean Graph. In this case, we will show that the algebraic
cycle matroid $M_{AC}(G)$ is FP-orientable. Clearly, we can (and will) assume without loss of
generality that $G$ is connected. We will rely on the fact that digraphs inherently possess a property
that is the graph-theoretical analog of the Farkas property. This is a well-known fact for finite
digraphs, see for instance \citep[Section 2.2]{bachemlpd}; for the proof in the more general case,
we will need the following lemma.

\begin{lem}\label{lem:Nibble}
  Let $T$ be a rayless tree which is a subgraph of $G$. Then $U_{V(T)}$ is a nibble.
\end{lem}
\begin{proof}
  Suppose not for a contradiction. So there is some ray $R$ in $G[V(T)]$. Since
  $U_{V(T)} = U_{V \setminus V(T)}$, there must also be some ray $S$ in $G[V \setminus V(T)]$. Since
  $T$ is rayless, by the Star-Comb Lemma \citep[Lemma 8.2.2]{graphtheory} there must be an infinite
  star in $T$ with all leaves on $R$. Then this star, together with $R$, $S$, and a path joining them,
  gives a subdivision of the Bean Graph in $G$, which is the desired contradiction.
\end{proof}

Recall that an {\em arborescence} with root $r$ is a directed graph in which there is a unique directed
walk from $r$ to any other vertex. A {\em reverse arborescence} with root $r$ is a directed graph in
which there is a unique directed walk from any other vertex to $r$. With these technical preparations
we obtain the following theorem which is essentially the Farkas property for the algebraic cycle matroid.

\begin{thm}\label{thm:FPForDigraphs}
  Let $vw$ be an edge of a digraph $D$. Then there is either a directed cycle, a directed double ray or
  a directed minimal (non-empty) nibble containing $vw$.
\end{thm}
\begin{proof}
  Let $T$ be a maximal arborescence with root $w$ which is a subdigraph of $D$ and $S$ a maximal reverse
  arborescence with root $v$ which is a subdigraph of $D$ (these exist by Zorn's Lemma). If $S$ meets $T$,
  then $vw$ lies on a directed cycle. Otherwise, if each of $S$ and $T$ includes a ray, then $vw$ lies on
  a directed double ray. So suppose without loss of generality that $T$ is rayless. Then by
  Lemma~\ref{lem:Nibble} the set $U_{V(T)}$ is a directed nibble containing $vw$. Since this nibble is a
  disjoint union of directed minimal non-empty nibbles, we are done.
\end{proof}

\begin{exmp}\label{exmp:AlgCycleMatsAreFPOriMats}
Bowler and Carmesin show in \citep[Subsection 5.3]{exclminorinfmat} that the algebraic cycle matroid
$M_{AC}(G)$ is signable. From Example~\ref{exmp:TameRegMatsAreOrthOriMats} we know that such a signing
induces a pair $\CC, \CC^*$ of circuit and cocircuit signatures of $M_{AC}(G)$ that provides an
orthogonal orientation $M_{AC}(G)$. Now, let $N$ be a minor of $M_{AC}(G)$ and $A \subseteq E(N)$. In
particular, the sets $\SC^\CC(N)$ and $\SC^{\CC^*}(N)$ then induce circuit and cocircuit signatures
on $N$ which are compatible with the signing of $N$ induced by the signing of $M_{AC}(G)$ (see
Remark~\ref{rem:SSIndByCircSig}). By Theorem~\ref{thm:FPForDigraphs} the pair
$\SC^\CC(N), \SC^{\CC^*}(N)$ has the Farkas property (FP); since this property is not lost under the
action of reorientation, the pair ${_{-A}\SC^\CC(N)}, {_{-A}\SC^{\CC^*}(N)}$ has it as well. Thus
$M_{AC}(G)$ is FP-orientable.
\end{exmp}

\begin{rem}\label{rem:NonCoFinitaryFPOriMatsExist}
  So far, the examples of FP-orientable matroids included only finite or finitary/cofinitary matroids.
  Algebraic cycle matroids, however, can have both infinite circuits and infinite cocircuits (see
  for instance \citep[Example 2.7]{infmatroids}).
\end{rem}

\subsubsection{Orthogonally Oriented Matroids that are not FP-Oriented Matroids}
The matroids and pairs of circuit and cocircuit signatures that we encountered in
Examples~\ref{exmp:OrthOriMAtWithSSCEAndNonSSCECMinor} and \ref{exmp:CMWithoutFarkas} provide
examples of orthogonally oriented matroids that are not FP-oriented matroids. To see why this
is the case, we examine both of them again and note that this additionally shows that the class
of FP-oriented matroids is a proper subclass of the class of orthogonally oriented matroids.

\begin{exmp}\label{exmp:OrthOriMAtWithSSCEAndNonSSCECMinorCont}
  The orthogonally oriented matroid $(\UC_{3,Q}, \CC_{\UC_{3,Q}}, \CC^*_{\UC_{3,Q}})$ from
  Example~\ref{exmp:OrthOriMAtWithSSCEAndNonSSCECMinor} is not an ordinary orthogonally oriented
  matroid because both of its circuit and cocircuit signatures have the strong circuit elimination
  property (CE). However, it is not an FP-oriented matroid since it has a minor which is not. As
  we have seen, the cocircuit signature induced by $\CC^*_{\UC_{3,Q}}$ on the contraction minor
  $\UC_{3,P}$ does not have the strong circuit elimination property. Thus by
  Corollary~\ref{cor:FPOrientationsImplSSCE} this minor is not an FP-oriented matroid. Another
  way to derive the same conclusion without examining minors is to show that the following two
  properties hold: no signed circuit of $\UC_{3,Q}$ is positive/negative and there exists an element
  of $Q$ that is neither contained in the support of a positive nor in the support of a negative cocircuit.
  First, consider an arbitrary signed circuit $C$ and let $\supp{C} = \{a,b,c,d\}$. The lines $a,b,c,d$
  give raise to four lexicographically positive points $a',b',c',d'$ on $S^2$. Since $Q$ is free, it is
  easy to see that there must exist a pair of elements, say $a',b'$, such that the remaining two points
  $c',d'$ lie on the same side of the great circle through $a',b'$. Now, let $Q_{a,b}$ be one of the two
  possible signed cocircuits with support $Q \setminus \{a,b\}$. It follows that $Q_{a,b}(c) = Q_{a,b}(d)$,
  i.e.\ $C$ cannot be positive/negative. To see that the second property holds, assume that $Q_{a,b}$ is a
  positive/negative cocircuit of $\UC_{3,Q}$ for some elements $a,b \in Q$. This implies that the plane
  spanned by $ab$ is the plane $\{ (x,y,z) \in \RR^3\ |\ x = 0 \}$ because otherwise $Q_{a,b}$ would not be
  positive/negative since $Q$ is $S^2$-dense. Thus either there is no positive/negative cocircuit or neither
  of the elements $a$ and $b$ is contained in a positive/negative cocircuit since $Q$ is free.
\end{exmp}

\begin{exmp}\label{exmp:CMWithoutFarkasCont}
  In Example~\ref{exmp:CMWithoutFarkas}, we have examined a finitary matroid $M$ on $E = \NN$
  and its dual $M^*$. We have observed that the circuits and cocircuits of $M$ can be signed in
  such a way that the resulting circuit and cocircuits signatures $\CC$ and $\CC^*$ both have
  the strong circuit elimination property (CE). According to
  Proposition~\ref{prop:SignedCircElimImplOrthAxioms}, the triple $(M, \CC, \CC^*)$ is thus an
  orthogonally oriented matroid. However, the pair $\CC, \CC^*$ has other interesting additional
  properties which we will examine in more detail below.\\
  First, the pair $\CC, \CC^*$ neither has the Farkas property (FP) nor the 4-painting property (4P):
  Consider the element $e = 1 \in E$ and note that $\CC$ contains no positive or negative circuits
  at all. On the other hand, $U_1$ and $-U_1$ are the only positive and negative cocircuits of $M$.
  Thus neither (FP) nor (4P) are satisfied for $e$ and therefore $(M, \CC, \CC^*)$ is not an FP-oriented
  matroid.\\
  Second, consider what happens when we replace the sets $\SC^\CC(N)$ and $\SC^{\CC^*}(N)$ from
  Definition~\ref{defn:SSIndByCircSig} by the sets $\SC^\VC(N)$ and $\SC^\WC(N)$ where $\VC$
  and $\WC$ denote the set of vectors and covectors of $M$ according to $\CC$ and $\CC^*$,
  respectively. Let $F, G \subseteq E$ such that $F \cap G = \emptyset$. For every minor
  $N = M\!/F\!\setminus\!G$ of $M$ we obtain two sets
  \[
    \SC^\VC(N) = \left\{ X|_{E(N)}\ \left|\ X \in \VC,\
                                            \supp{X} \subseteq E(N) \cup F \right.\right\},
  \]
  \[
    \SC^\WC(N) = \left\{ Y|_{E(N)}\ \left|\ Y \in \WC,\
                                            \supp{Y} \subseteq E(N) \cup G \right.\right\}.
  \]
  Similar to Remark~\ref{rem:SSIndByCircSig}, we note that if $V \in \SC^\VC(N)$, then $\supp{V}$
  is a scrawl of $N$ by Corollary~\ref{cor:StructScrawlsOfMinor}. Dually, if $W \in \SC^\WC(N)$,
  then $\supp{W}$ is a coscrawl of $N$. Additionally, $(\SC^\VC(N), \SC^\WC(N))$ is a
  {\em minor of $(\VC, \WC)$} in the sense of Definition~\ref{defn:BKOriMats}.\\
  We claim that the pair ${_{-A}\SC^\VC(N)}, {_{-A}\SC^\WC(N)}$ has the Farkas property (FP) for every
  minor $N$ of $M$ and all subsets $A \subseteq E$. We first consider the case $N = M$ and the following
  four subcases with respect to the possible shapes of $A$.
  \begin{enumerate*}
    \item If $A = \emptyset$, then every element of $E$ is contained in the support of the positive covector
      $-U_1 \circ U_2$ and there are no positive/negative vectors.
    \item If $A = E$, then every element of $E$ is contained in the support of the positive covector
      $_{-A}U_1 \circ -(_{-A}U_2)$. Again, there are no positive/negative vectors.
    \item If $A = \{ j\ |\ j \leq i \}$ or $A = \{ j\ |\ j > i \}$ for some $i \in E$, then
      $_{-A}U_i \circ {_{-A}U_{i+1}}$ or its negation is a positive covector whose support covers $E$.
      On the other hand, the reorientation on $A$ does not introduce any positive/negative vectors.
    \item If the shape of $A$ is different from the ones examined in the previous three cases, then
      $_{-A}S^{\WC}(M)$ does not contain any positive/negative covectors: if $A$ contains more than one
      element, then at least one of the two "legs" of every cocircuit intersects both $A$ and $E \setminus A$.
      Otherwise $A = \{ j \}$ for some $j > 1$. Then $U_j$ and $-U_j$ are the only two cocircuits not
      satisfying this condition but neither of those cocircuits is positive/negative. So in both cases
      there cannot be any positive/negative covectors. It remains to show that every $e \in E$ is contained
      in a positive circuit---and thus a positive vector---with respect to the reorientation on $A$.
      In the following, if $i, j \in E$ are two distinct elements, then we say that $i$ is {\em to the left of}
      $j$ and $j$ is {\em to the right of} $i$, respectively, if $i < j$. Let $e \in E$. If $e \in A$,
      then we have to consider the following three cases.
      \begin{enumerate*}
        \item There are no elements to the left of $e$ which are not in $A$. Since $A \ne E$ and $A$ does
          not have the shape $\{ j\ |\ j \leq i \}$ for some $i \in E$, there are $m \in E \setminus A, n \in A$
          such that $e < m < n$. Then $-_{-A}C_{emn}$ is a suitable positive circuit.
        \item There are elements to the left of $e$ which are not in $A$ but no element to the left of any
          of those elements is in $A$. We pick one of the elements to the left of $e$ which are not in $A$
          and call it $m$. Since $A$ does not have the shape $\{ j\ |\ j > i \}$ for some $i \in E$, there
          exists an $n \in E \setminus A$ such that $e < n$. Then $_{-A}C_{men}$ is a suitable positive circuit.
        \item There is an element $n$ to the left of $e$ which is not in $A$ and there is an element $m$ to
          the left of $n$ which is in $A$. Then $-(_{-A}C_{mne})$ is a suitable positive circuit.
      \end{enumerate*}
      If $e \notin A$, a similar case distinction shows that a suitable positive circuit containing $e$ can
      also be chosen.
  \end{enumerate*}
  This concludes the proof for the case $N = M$. Now, let $F, G \subseteq E, F \cap G = \emptyset$,
  and $N = M\!/F\!\setminus\!G$ be a proper minor of $M$. It suffices to consider the following
  three cases.
  \begin{enumerate*}
    \item $F = \emptyset$: With similar arguments as in the case of $M = N$ above we conclude that the
      pair ${_{-A}\SC^\VC(N)}, {_{-A}\SC^\WC(N)}$ has the Farkas property (FP) for
      all subsets $A \subseteq E(N)$.
    \item $F = \{ i \}$ for an $i \in E$: In this case,
      $\SC^\WC(N)$ collapses to $\{ \pm U_i|_{E(N)} \}$. If $_{-A}U_i$ is positive or negative, then
      (FP) holds trivially for the pair ${_{-A}\SC^\VC(N)}, {_{-A}\SC^\CC(N)}$.
      Otherwise, there exists an $f \in E(N)$ for every $e \in E(N)$ such that
      $(_{-A}U_i)(e) = -(_{-A}U_i)(f)$. Thus it is possible to choose a suitable $C \in \CC$
      such that $\supp{C|_{E(N)}} = \{ e, f \}$ and $_{-A}C|_{E(N)}$ is positive.
    \item $|F| \geq 2$: In this case, $\SC^\WC(N)$ is empty. If $e \in E(N)$, then a suitable
      $C \in \CC$ such that $\supp{C|_{E(N)}} = \{e\}$ can be chosen by picking elements from $F$
      accordingly. Thus (FP) is trivially satisfied for any reorientation of the pair
      $\SC^\VC(N), \SC^\WC(N)$.
  \end{enumerate*}
  Now, if the ground set of $M$ was finite, this observation would imply that $(M, \CC, \CC^*)$ is an
  oriented matroid; in particular, by the {\em Composition Theorem} \citep[Theorem 5.36]{bachemlpd},
  every vector and covector of $M$ could then be decomposed into conforming signed circuits and cocircuits,
  respectively. The latter, however, does in general not hold under similar conditions in the infinite
  case: note that the positive covector $-U_1 \circ U_2$ of $M$ cannot be decomposed into conforming signed
  cocircuits. This demonstrates the following: First, the triple $(M, \CC, \CC^*)$ is an example of an
  infinite orthogonally oriented matroid that does not allow the decomposition of covectors into conforming
  signed cocircuits even though it satisfies the seemingly strong condition that the pair
  ${_{-A}\SC^\VC(N)}, {_{-A}\SC^\WC(N)}$ has the Farkas property (FP) for every minor $N$ of $M$ and all
  subsets $A \subseteq E$. Second, a countably infinite ground set is sufficient to show this difference
  from the finite case. Finally, it is does not suffice to require that the pair $\SC^\VC(M), \SC^\WC(M)$
  has the Farkas property (FP) to conclude that the pair $\CC, \CC^*$ has it as well.
\end{exmp}

\begin{rem}
  Examples~\ref{exmp:OrthOriMAtWithSSCEAndNonSSCECMinorCont} and \ref{exmp:CMWithoutFarkasCont}
  additionally show that there cannot exist a characterization of tame FP-orientable matroids
  similar to the characterization of tame orthogonally orientable matroids given by
  Theorem~\ref{thm:OrthCharByMinors}.
\end{rem}

\section{Classes of (Oriented) Matroids: Visual Overview}\label{sec:OverviewClassesOfOriMats}
Throughout the text we have encountered numerous classes of matroids: finite, (co)finitary,
regular orthogonally oriented and/or FP-oriented. In this section, we provide a visual overview
of those classes in the shape of several diagrams. Common to these diagrams is the question how
the different classes of matroids relate to each other in terms of subclass relationships or
overlapping.\\

In the first diagram, Figure~\ref{fig:ClassesOfMatroidsWithoutFP}, FP-oriented matroids are
left out since their distribution clutters the visualization a bit. Note that we did not present
concrete examples of wild orthogonally oriented matroids. Since we conjecture that this class
actually exists, we have included it in the visualization, but choose to represent it by a dashed
line.\\

The second diagram, Figure~\ref{fig:ClassesOfMatroidsWithFP}, augments the first one by taking
the class of FP-oriented matroids into account to provide a complete overview. For the sake of
clarity, this class is spread out over the diagram a little more than necessary. (It would also
be possible to depict several parts of its distribution as one large, connected area; however,
this would in turn make it more difficult to explicitly refer to each part in the annotations
to the diagram in Table~\ref{tab:ExpColorsDiaWithFP}.) As in the case of the class of wild
orthogonally oriented matroids, we represent those subclasses of FP-oriented matroids whose
existence we consider only as a conjecture by a dashed line.\\

Finally, in Figure~\ref{fig:ClassesOfMatroidsNotes} together with Table~\ref{tab:ExpColorsDiaWithNotes},
we indicate the passages in the text where explicit reference is made to the respective parts
of the diagram from Figure~\ref{fig:ClassesOfMatroidsWithFP}.

\begin{figure}[ht]
  \centering
  \includegraphics[width=0.95\textwidth]{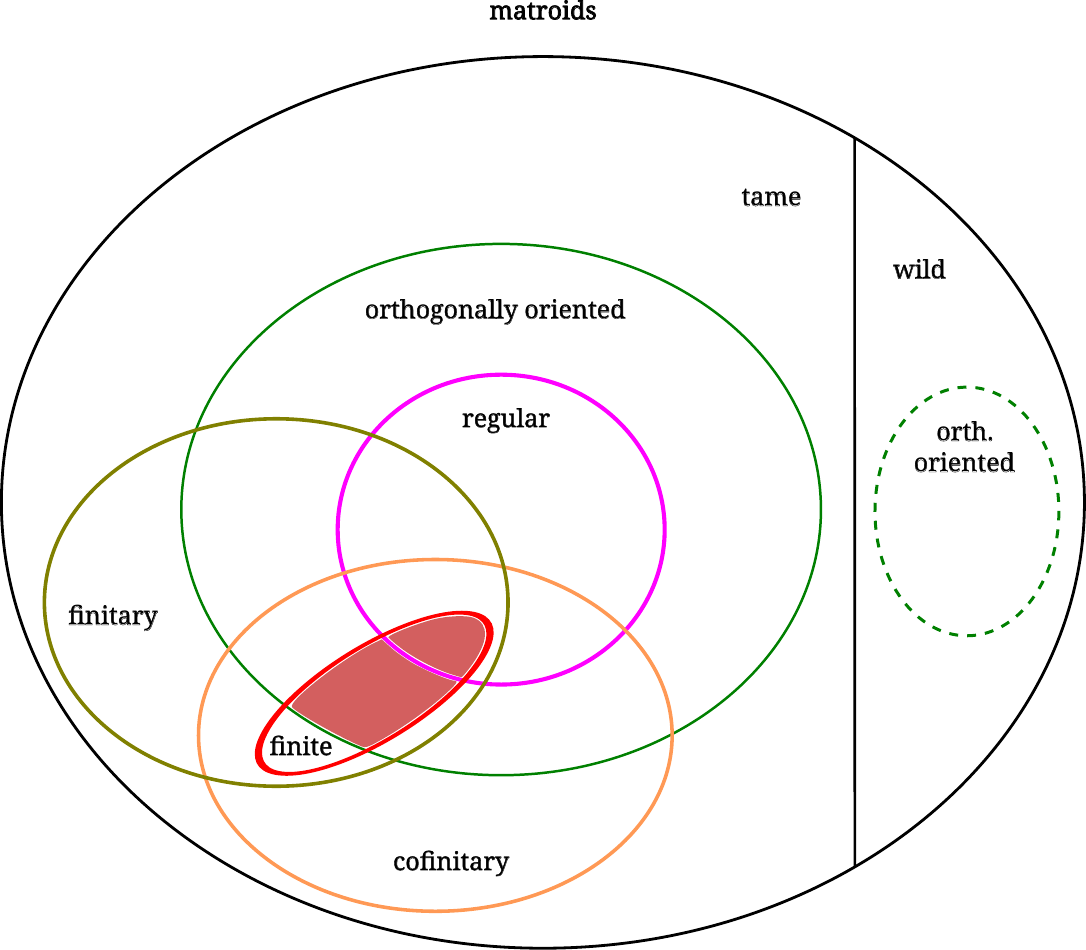}
  \caption{Classes of matroids and their (subclass) relationships (FP-oriented matroids omitted)}
  \label{fig:ClassesOfMatroidsWithoutFP}
\end{figure}

\begin{table}[ht]
  \centering
  \begin{tabular}{ll}
    \toprule
    \textbf{Color}                                                  & \textbf{Class}\\
    \midrule
    \textcolor{all-matroids}{\rule{0.5cm}{0.2cm}} black             & Matroids\\
    \textcolor{orth-ori}{\rule{0.5cm}{0.2cm}} green (solid)         & Tame orthogonally oriented matroids\\
    \textcolor{orth-ori}{\rule{0.5cm}{0.2cm}} green (dashed)        & Wild orthogonally oriented matroids {\em (conjectured)}\\
    \textcolor{regular}{\rule{0.5cm}{0.2cm}} pink                   & Regular matroids\\
    \textcolor{finitary}{\rule{0.5cm}{0.2cm}} ocher                 & Finitary matroids\\
    \textcolor{cofinitary}{\rule{0.5cm}{0.2cm}} orange              & Cofinitary matroids\\
    \textcolor{finite}{\rule{0.5cm}{0.2cm}} red                     & Finite matroids\\
    \textcolor{finite-oriented}{\rule{0.5cm}{0.2cm}} dim red (area) & Finite oriented matroids\\
    \bottomrule
  \end{tabular}
  \captionof{table}{Explanation of the colors used in Figure~\ref{fig:ClassesOfMatroidsWithoutFP}}
\end{table}

\begin{figure}[ht]
  \centering
  \includegraphics[width=0.95\textwidth]{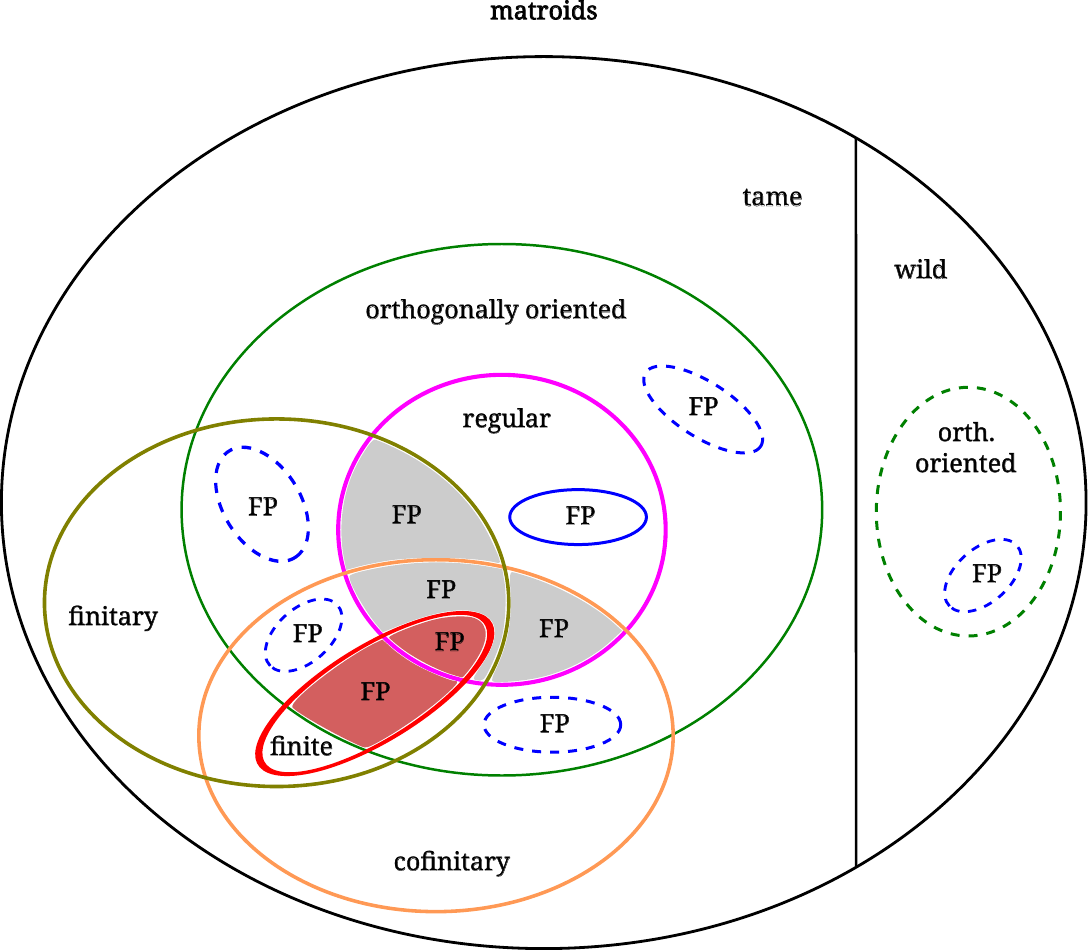}
  \caption{Conjectured distribution of the class of FP-oriented matroids}
  \label{fig:ClassesOfMatroidsWithFP}
\end{figure}

\begin{table}[ht]
  \centering
  \begin{tabular}{lp{0.6\textwidth}}
    \toprule
    \textbf{Color}                                                  & \textbf{Description}\\
    \midrule
    \textcolor{finite-oriented}{\rule{0.5cm}{0.2cm}} dim red (area) & Finite FP-oriented matroids (i.e. ordinary oriented matroids)\\
    \textcolor{fin-cofin-regular}{\rule{0.5cm}{0.2cm}} grey (area)  & Infinite regular FP-oriented matroids that are finitary, cofinitary or both\\
    \textcolor{reg-fp-ori}{\rule{0.5cm}{0.2cm}} dim blue (solid)    & Infinite regular FP-oriented matroids that are neither finitary nor cofinitary\\
    \textcolor{conj-fp-ori}{\rule{0.5cm}{0.2cm}} blue (dashed)      & Remaining subclasses of the class of FP-oriented matroids {\em (conjectured)}\\
    \bottomrule
  \end{tabular}
  \captionof{table}{Explanation of the (additional) colors used in Figure~\ref{fig:ClassesOfMatroidsWithFP}}
  \label{tab:ExpColorsDiaWithFP}
\end{table}

\begin{figure}[ht]
  \centering
  \includegraphics[width=0.95\textwidth]{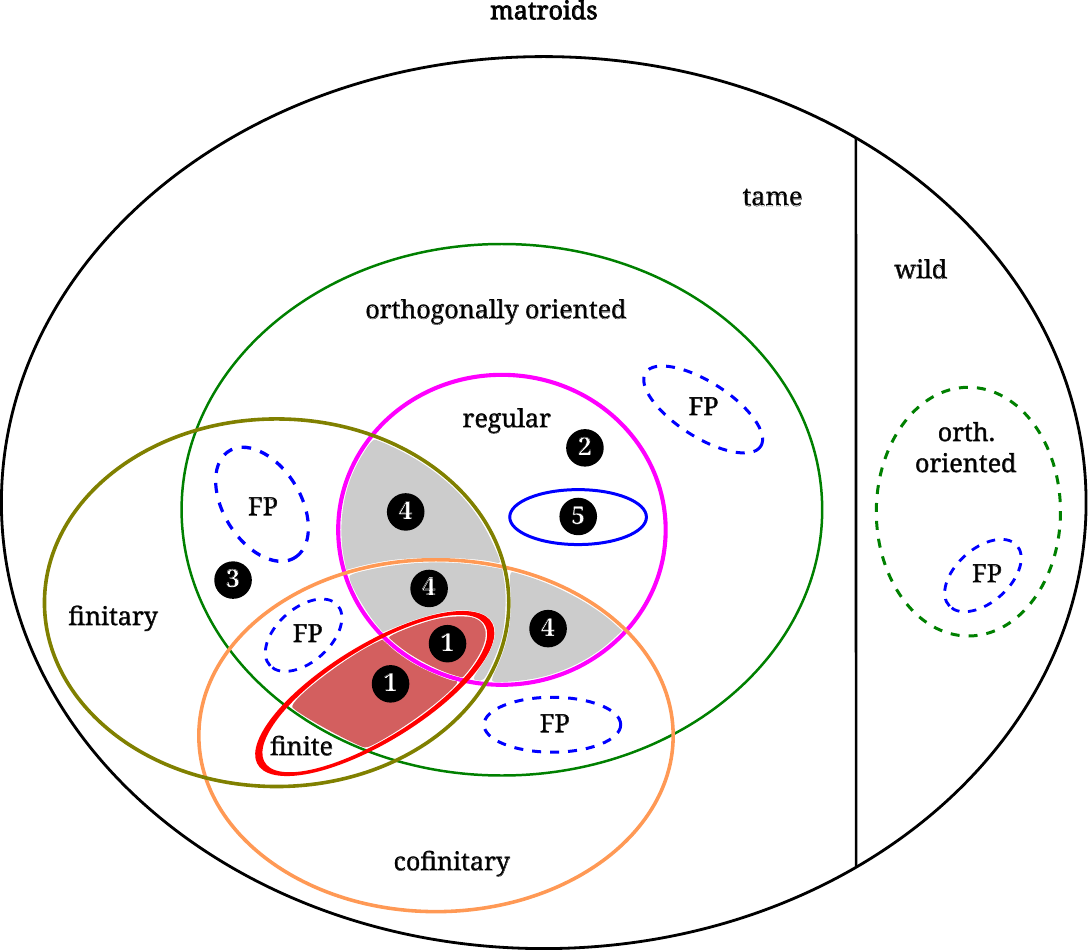}
  \caption{Which parts of the diagram from Figure~\ref{fig:ClassesOfMatroidsWithFP} are explicitly considered
           throughout the text}
  \label{fig:ClassesOfMatroidsNotes}
\end{figure}

\begin{table}[ht]
  \centering
  \begin{tabular}{rp{0.6\textwidth}>{\RaggedRight}p{0.17\textwidth}}
    \toprule
    \textbf{No.} & \textbf{Description}                                                              & \textbf{Reference}\\
    \midrule
    1            & Finite oriented matroids are orthogonally orientable as well as FP-orientable     & Example~\ref{exmp:FinOriMatsAreOrthOriMats},
                                                                                                       Example~\ref{exmp:FinOriMatsAreFPOriMats}\\
    2            & Regular matroids are orthogonally orientable                                      & Example~\ref{exmp:TameRegMatsAreOrthOriMats}\\
    3            & There are finitary orthogonally oriented matroids that are not FP-orientable      & Example~\ref{exmp:OrthOriMAtWithSSCEAndNonSSCECMinorCont},
                                                                                                       Example~\ref{exmp:CMWithoutFarkasCont}\\
    4            & (Co)Finitary regular matroids are FP-orientable                                   & Example~\ref{exmp:CofinRegMatsAreFPOriMats}\\
    5            & There are algebraic cycle matroids that are neither finitary nor cofinitary       & Example~\ref{exmp:AlgCycleMatsAreFPOriMats},
                                                                                                       Remark~\ref{rem:NonCoFinitaryFPOriMatsExist}\\
    \bottomrule
  \end{tabular}
  \captionof{table}{List of references provided by Figure~\ref{fig:ClassesOfMatroidsNotes}}
  \label{tab:ExpColorsDiaWithNotes}
\end{table}

\clearpage

\bibliographystyle{plainnat}
\bibliography{on-the-pos-def-inf-ori-mats-bib}

\end{document}